\documentclass{class_for_arxiv}
\RUNAUTHOR{S. Allmeier and N. Gast}

\usepackage{amsmath,amssymb,amsfonts,amsbsy}

\usepackage{fix-cm}
\usepackage{lmodern}
\usepackage{anyfontsize}

\usepackage{appendix}

\usepackage[dvipsnames]{xcolor}
\usepackage[%
  colorlinks = true, 
  citecolor  = RoyalBlue,
  linkcolor  = RoyalBlue,
  urlcolor   = RoyalBlue,
  unicode,
  ]{hyperref}
\usepackage{url}

\usepackage[small, margin=1cm]{caption}

\usepackage{natbib}
 \bibpunct[, ]{(}{)}{,}{a}{}{,}%
 \def\bibfont{\small}%

\usepackage{rotating}
\usepackage{fancyvrb}

\TheoremsNumberedThrough     
\JOURNAL{Stochastic Systems}

\EquationsNumberedThrough    

\usepackage{hyperref}

\usepackage{stackengine}

\usepackage{dsfont}
\usepackage{stmaryrd}

\usepackage{graphicx}
\graphicspath{{./plots/}{}{figures/}}

\usepackage{xcolor}
\usepackage{pifont}

\usepackage{units}

\usepackage{subcaption}

\usepackage{fancyvrb}
\fvset{fontsize=\normalsize}

\usepackage{multicol}

\usepackage{enumitem}

\usepackage{tikz}
\usetikzlibrary{shapes.geometric, arrows}
\usetikzlibrary{decorations.pathreplacing}

\usepackage{mymacros}

\begin{document}

\RUNTITLE{Accuracy of the Graphon Mean Field Approximation}

\TITLE{Accuracy of the Graphon Mean Field Approximation for Interacting Particle Systems}

\ARTICLEAUTHORS{%
\AUTHOR{Sebastian Allmeier}

\AFF{Univ. Grenoble Alpes, Inria, CNRS, Grenoble INP, 38000 Grenoble, France \EMAIL{sebastian.allmeier@inria.fr}}

\AUTHOR{Nicolas Gast}

\AFF{Univ. Grenoble Alpes, Inria, CNRS, Grenoble INP, 38000 Grenoble, France \EMAIL{nicolas.gast@inria.fr}}

}




\ABSTRACT{

We consider a system of $N$ particles whose interactions are characterized by a (weighted) graph $G^N$. Each particle is a node of the graph with an internal state. The state changes according to Markovian dynamics that depend on the states and connection to other particles. 
We study the limiting properties, focusing on the dense graph regime, where the number of neighbors of a given node grows with $N$. We show that when $G^N$ converges to a graphon $G$, the behavior of the system converges to a deterministic limit, the graphon mean field approximation. We obtain convergence rates depending on the system size $N$ and cut-norm distance between $G^N$ and $G$. 
We apply the results for two subcases:
When $G^N$ is a discretization of the graph $G$ with individually weighted edges; when $G^N$ is a random graph obtained through edge sampling from the graphon $G$. In the case of weighted interactions, we obtain a bound of order $\mathcal{O}(1/N)$. In the random graph case, the error is of order $\mathcal{O}(\sqrt{\log(N)/N})$ with high probability. 
We illustrate the applicability of our results and the numerical efficiency of the approximation through two examples: a graph-based load-balancing model and a heterogeneous bike-sharing system.

}

\KEYWORDS{graphon, mean field approximation, interacting particle systems, random graph, bias bound}

\maketitle

\section{Introduction}

Mean field approximations are ubiquitously used in the study of large scale stochastic systems. Its mathematical foundations date back to the $'60$s and $'70$s with the seminal papers of \cite{kurtzSolutionsOrdinaryDifferential1970,kurtzLimitTheoremsSequences1971}, \cite{mckeanPropagationChaosClass1967} and \cite{normanMarkovProcessesLearning1972}. While originating from statistical physics, the mean field methodology has found application in many areas, such as communication networks \cite{leboudecGenericMeanField2007}, load balancing \cite{mitzenmacherPowerTwoChoices2001}, and the study of epidemics \cite{royRecentAdvancesModeling2023}. The fundamental idea of the mean field method is to represent the particle process by a Markovian state descriptor based on averaged quantities of the system. A quantity commonly used in this context is state occupancy processes, e.g., the averaged load of stations for a bike sharing system \cite{frickerIncentivesRedistributionHomogeneous2016} or the fraction of servers having at least a certain queue length for load balancing systems \cite{mitzenmacherPowerTwoChoices2001}. To apply the `classical' mean field method, it is crucial that the particles of the system are homogeneous and therefore interchangeable, ensuring the Markov property for the mentioned aggregate quantities. 
Its applicability to real world scenarios is often hindered by the intrinsically heterogeneous behavior of realistic models as heterogeneity breaks the quintessential exchangeability assumption. Therefore, it is necessary to keep track of the evolution of the whole set of particles in the system which makes the application of `classical' mean field methods complicated, if not infeasible.


To overcome these difficulties, mean field models which use graphons as a description of diversity of individuals and the variation of their interaction have been proposed \cite{budhirajaSupermarketModelGraphs2019,betWeaklyInteractingOscillators2024,vizueteGraphonBasedSensitivityAnalysis2020,keligerUniversalitySisEpidemics2023,bayraktarGraphonMeanField2023}. One example for where this approach finds widespread application are epidemic models for which graphons characterize the epidemic propagation \cite{delmasIndividualBasedSIS2023,decreusefondLargeGraphLimit2012,aurellFiniteStateGraphon2022,royRecentAdvancesModeling2023}. For many settings, results show that deterministic graphon-based models are limiting processes for sequences of graph-based interacting and densely connected particle systems. 
While the use of the graphon mean field model is well justified, the bias and convergence speed in its application can vary immensely depending on the convergence properties of the considered graph sequence. Surprisingly, the literature concerning the issue of convergence rates and accuracy is relatively sparse. This paper contributes to filling this gap as we derive generic bounds on the bias for the graphon approximation for finite-sized and densely connected interacting particle systems.




\subsection{Contributions}
We provide bias bounds for the graphon mean field approximation for finite-sized systems consisting of $N\in \N$ interacting particles for which the graph $G^N$ models the connection of the population. Our results show that it is possible to derive bias bounds which largely depend on the convergence properties of the graph sequence $G^N$ and of its limiting graphon $G$.
To be more precise, we start from a stochastic interacting particle model of finite size $N$, where each individual $k$ is characterized by a time-varying state $S_k(t)$. The connection of the particle to the population is given by the edges $(k,l)$ of a graph $G^N$. Based on this description, we construct a (deterministic) integro-differential equation based on the graphon $G$, and show that it has a unique solution $\bODE(t)$ 
that we call the graphon mean field approximation. This differential equation is constructed such that for fixed $N$, $\ODE_{k/N,s}(t)$ approximates the probability of particle $k\in[N]$ to be in state $s$ at time $t$. Denoting $\mathbb{P}(\SN_k(t) = s)$ this probability, our main result shows that
\begin{align*}
    \mathbb{P}(S^{G^N}_k(t) = s) = \ODE_{k/N,s}(t) + \mathcal{O}(\frac1N + \opnorm{G^N-G}_{L_2}),
\end{align*}
where $\opnorm{G^N-G}_{L_2}$ is the $L_2$ distance of the step graphon representation of $G^N$ and the graphon $G$, which is equivalent to the distance implied by graphon typical cut norm, see Theorem \ref{thrm:approximation_accuracy} for a precise description. 

To show the extent of the result, we consider two specific graph sampling strategies for $G^N$. In the first case (which we denote \emph{deterministic sampling}), $G^N$ is the discretization of the limiting graphon $G$, with $G^N_{kl} := G_{k/N, l/N} \in(0,1]$ being the strength of interaction between two particles. We consider this case as it illustrates how our result can be used to model 
In the second case (that we call \emph{stochastic sampling}), $G^N$ is a random graph generated from $G$, where an edge is present between two nodes $k$ and $\ell$ with probability $G(k/N,l/N)$. The second case corresponds to a generalization of \ER graph and stochastic block models to possibly non-uniform probabilities. 

Imposing some mild assumptions such as a piecewise Lipschitz condition on the graphon $G$ or symmetry of $G$, in the case of stochastic sampling, we can bound the distance between the graph $G^N$ and the graphon $G$ by:
\begin{align*} \opnorm{G^N-G}_{L_2} =
    \begin{cases} 
    \mathcal{O}(N^{-1}). &\text{ (Deterministic Sampling)} \\
    \mathcal{O}(\sqrt{\frac{\log(N)}{N}}) ~ \text{w.h.p.,} & \text{ (Stochastic Sampling)}
    \end{cases} 
\end{align*}
Here, with high probability (w.h.p) means that the right-hand side in case of the stochastic sampling holds with probability at least $1 - 2/N$. The precise assumptions and results are given in the Corollaries \ref{corollary:graphon_results} and \ref{corollary:het_approx_results}.
We point out that the main result applies to the considered cases but is not limited to those. It is possible to consider other graph sampling strategies that satisfy our assumptions and for which the $L_2$ distance between the step graphon representation of $G^N$ and the graphon $G$ can be bounded. 
To illustrate our results and emphasize their applicability, we provide two examples, one each for the stochastic and deterministic sampling method.
Our first example considers a load-balancing system with stochastically drawn connections between servers. Jobs in the load balancing system arrive at a server site, where each server similarly acts as a dispatcher and keeps or forwards the job according to the JSQ($2$) policy based on its connected neighbors. The second example illustrates the application of deterministic sampling for a bike sharing system, with particles being stations and the graphon determining the popularity of the stations.
In both cases, a simple discretized version of the integro-differential equation already yields precise estimations of the system dynamics while the numerical complexity of the approximation only slightly increases compared to the homogeneous case.

\subsection{Organization} 
The paper is organized as follows. In Section \ref{sec:particle_system} we introduce the heterogeneous particle model. Section \ref{sec:connectivity_function} defines graphon, related sampling methods, and related preliminary results. Our definition of the approximation, the main results, and the proofs are displayed in Section \ref{sec:main_results}. Last, in Section \ref{sec:numerical_experiments} we present the two numerical examples. 

\subsection{Related Work}
\paragraph{Dynamical Systems on Random Graphs} 

In recent years there has been growing interest in the behavior of interacting particles that are interconnected by an underlying graph topology, for example, \cite{bayraktarMeanFieldInteraction2021,betWeaklyInteractingOscillators2024,bhamidiWeaklyInteractingParticle2019,abbeCommunityDetectionStochastic2018} and previously mentioned references. For a general introduction to the topic of random graph networks and limiting graphon functions for dense graphs, we refer to the works \cite{vanderhofstadRandomGraphsComplex2017,lovaszLargeNetworksGraph2012}. 
The majority of papers focus on dense graphs, having edges of order $N^2$ such as \ER type graphs or graphs generated by stochastic block models. Additional related work can be found in the game theoretic setting for graphon mean field games, see for example \cite{cainesGraphonMeanField2021,aurellFiniteStateGraphon2022}. In the `not-so-dense' setting available results are more limited, with some newer references being \cite{bayraktarGraphonMeanField2023,delmasIndividualBasedSIS2023}. In the case of sparse graphs, such as $d$-regular graphs or random geometric graphs, the typical mean field methods break down as they fail to capture the importance of the spatial graph structure and its implications on the local dynamics of particles. Some recent works in this setting include \cite{gangulyHydrodynamicLimitsNonMarkovian2022,ramananMeanfieldLimitsAnalysis2022,gangulyNonMarkovianInteractingParticle2022}.

\paragraph{Load Balancing on Graphs} 
Our load balancing example is inspired by the recent works of \cite{ruttenMeanfieldAnalysisLoad2023,zhaoOptimalRateMatrixPruning2024,zhaoExploitingDataLocality2024,budhirajaSupermarketModelGraphs2019}. Here, the authors study a variety of load balancing systems with dynamics based on compatibility or locality constraints, which give rise to intricate connectivity between dispatchers and servers. While not directly transferable into the framework of this paper, the authors similarly deal with graph-based systems and limit approximations that are strongly related to the ones our framework suggests. Note that the techniques developed in these papers are very model-specific and allow for transitions to depend on the states of multiple queues, whereas our approach aims at deriving results for a more general framework for transition rates with the restriction to the case of pair interactions.

\paragraph{Generator and Stein's Method} For our proofs, we adapt techniques used in \cite{gastRefinedMeanField2017,allmeierMeanFieldRefined2022,gastExpectedValuesEstimated2017}, which in turn rely on the use of Stein's method \cite{steinApproximateComputationExpectations1986}. The method is used to estimate and bound the distance between two random variables through their respective generators. Since the works \cite{bravermanSteinMethodSteadystate2017c,bravermanSteinMethodSteadystate2017b} Stein's method has seen an increase in the stochastic network community and is an actively evolving area. 


\section{Heterogeneous Network Particle System} \label{sec:particle_system}

\subsection{The Interaction Model}

We consider particle systems with $N\in\N$ interacting particles. Connections between particles are characterized by a (possibly weighted) adjacency matrix $G^N \in[0,1]^{N \times N}$, with $G^N_{kl}$ indicating the connection strength between particle $k$ and $l$.  Each of the particles has a finite state space $\calS$ where the state of the $k$-th object at time $t\geq0$ is denoted by $\SN_k(t)$. As we see later, $G^N$ can correspond to a random graph for which $G^N_{kl}\in\{0,1\}$ indicates the presence or absence of an undirected edge between the particles $k$ and $l$ or can be an arbitrary weight matrix, see Section~\ref{sec:connectivity_function}. 
The state of the whole system is denoted by $\bSN(t) := \bigl( \SN_{1}, \dots , \SN_{N} \bigr)(t) \in \calS^N$. We assume that the process $\bSN:= (\bSN(t))_{t\geq0}$ is a continuous time Markov chain (CTMC) with the dynamics of the system  described as in the following. 
Each particle $k\in [N]$ changes its state from $s_k$ to $s_k'\in \calS$ in one of the two ways:
\begin{enumerate}[leftmargin=2.5cm]
	\item[(Unilateral)] 
    The change of state occurs at rate $r_{k, s_k \to s_{k}'}^{N, \text{uni}}$ independently of the other particles. 
	\item[(Pairwise)] 
    The change of state is triggered by another particle $l\in[N]$ that is in state $s_l\in \calS$.  This occurs at rate $r_{k,l, s_k \to s_k', s_{l}}^{N, \text{pair}} G^N_{kl}/N$.
\end{enumerate}
Note that the rate functions are assumed to be heterogeneous, i.e., they can depend on the items $k$ and $l$. The rates have a $1/N$ factor, as each particle can potentially interact with all $N-1$ remaining particles. Hence, our condition implies that the total rate of transitions of the particle system is of order $O(N)$ and that the transition rates for all particles are of the same order. As we will further discuss in Example \ref{subsec:load_balancing_example}, our results can also be used if the scaling factor depends not on the system size $N$ but on the node degrees.
We further want to point out that we restrict our framework to the interaction of two particles, which can be utilized to model many relevant interacting particle systems on graphs such as \emph{e.g.}, epidemic spreading, power-of-two-choices load-balancing, or bike sharing systems. It is nonetheless possible to use the same underlying approach to extend the framework and results to interactions of higher order. This, however, comes at the cost of increasingly cumbersome notations and limited added theoretical insight. 



\subsection{The Binary State Representation}

In order to ease computations and definitions, we will use a binary representation of the state based on indicator functions. We denote the new representation by $\bXN = ( \XN_{k,s}(t) )_{\substack{k\in[N] \ s\in\calS \\ t\geq 0}}$, where 
\begin{align*} 
    \XN_{k,s}(t) := \mathds{1}_{\{\SN_k(t) = s\}} := 
    \begin{cases}
        1 &\text{if object $k$ is in state $s$ at time $t$}, \\
        0 &\text{otherwise.}
    \end{cases}
\end{align*}
The space of attainable states is denoted by $\calX^N \subset \{ 0, 1\}^{N \times \calS}$.


While this representation is less compact than the original, it allows for an easier definition of transition rates as well as the definition of the mean field approximation. Denote by $\BFe^N_{k,s}$ a matrix of size $N\times \abs{\calS}$ whose $(k,s)$ component is equal to 1, all other entries being zero. For each $k\in[N]$, and $s_k,s_k'\in\calS$, $\XN$ jumps to $\XN+\BFe_{k,s_k}^N-\BFe_{k,s_k'}^N$ at rate
\begin{align}
    X_{k,s_k}r_{k, s_k \to s_{k}'}^{N, \text{uni}} +  X_{k,s_k}\sum_{l\in[N]}\sum_{s_{l}\in \calS} r_{k,l, s_k \to s_k', s_{l}}^{N, \text{pair}} \frac{G^N_{kl}}{N} X_{l,s_{l}}.
    \label{eq:transitions}
\end{align}
In the above equation, the first term of the rate corresponds to the unilateral transition of the particle $k\in[N]$ changing its state from $s_k$ to $s_k'$, as this transition occurs at intensity $r_{k, s_k \to s_{k}'}$ if particle $k$ is in state $s_k$ (\emph{i.e.}, $X_{k,s_k} = 1$). The second term describes the pairwise transitions leading to the state change of particle $k$ from $s_k$ to $s_k'$. Similar to the unilateral one, the transition can only happen if particle $k$ is in state $s_k$, represented in the rate by the prefactor $X_{k,s_{k}}$. The remaining factor corresponds to the interaction with other particles, expressed by the weighted sum over all other particles and their states. The intensity of the transition is scaled by $r_{k,l, s_k \to s_k', s_{l}}^{N,\text{pair}}$ and the connectivity of the particles is given by the connectivity matrix $G^N\in [0,1]^{N\times N}$.

\subsection{Drift of the System of Size \texorpdfstring{$N$}{N}}

By using the state representation $\bXN$, we define what we call the \emph{drift} of the system of $N$ particles as the expected change for $\bXN$ in state $\bX \in \calX\toN$
. It is equal to the sum of all possible transitions of the changes induced by this transition times the rate at which the transition occurs. We denote this quantity as $\bFN(\bX)$. By using Equation~\eqref{eq:transitions}, it is equal to:
\begin{align*}
    \bFN(\bX) &= \sum_{k\in[N],s_k,s_k'\in\calS}(\BFe^N_{k,s_k'}-\BFe^N_{k,s_k})\left[X_{k,s_k}r_{k, s_k \to s_{k}'}^{N,\text{uni}} +  X_{k,s_k}\sum_{l\in[N]}\sum_{s_{l}\in \calS} r_{k,l, s_k \to s_k', s_{l}}^{N,\text{pair}} \frac{G^N_{kl}}{N} X_{l,s_{l}}\right]
\end{align*}
The quantity $\bFN(\bX)$ is a vector-valued function of $\bX$. By reorganizing the above sum, $\FN_{k,s}(\bX)$ --its $(k,s)$ component-- is equal to
\begin{align*}
    \FN_{k,s_k}(\bX)=\sum_{s'\in\calS}\left[X_{k,s'}r_{k, s_k' \to s_k}^{N,\text{uni}}-X_{k,s}r_{k, s_k \to s_k'}^{N,\text{uni}} +  \sum_{l\in[N]}\sum_{s_{l}\in \calS} (X_{k,s_k'}r_{k,l, s_k' \to s_k, s_{l}}^{N,\text{pair}}-X_{k,s}r_{k,l, s_ \to s_k', s_{l}}^{N,\text{pair}} )\frac{G^N_{kl}}{N} X_{l,s_{l}}\right].
\end{align*}
In the following, it will be convenient to replace the above sum by matrix multiplications. To do so, let us denote by $X_k$ the vector $(X_{k,s})_{s\in\calS}$. The above equation can be written as:
\begin{align}
    \FN_{k,s}(\bX)= \bR^{N,\text{uni}}_{k,s}\bX_{k} + \bX^T_k\sum_{l\in[N]}  \bR^{N, \text{pair}}_{k,l,s} \bX_l\frac{G^N_{kl}}{N} ,
    \label{eq:drift_N}
\end{align}
where $\bR^{N,\text{uni}}_{k,s}$ is a row vector whose $s'$ component is $r_{k, s' \to s}^{N,\text{uni}}$ if $s'\ne s$ and $-\sum_{\tilde{s}} r_{k, s \to \tilde{s}}^{N,\text{uni}}$ for $s=s'$; and $\bR^{N,\text{pair}}_{k,l,s}$ is a matrix whose $(s',s_l)$ component is $r_{k,l, s' \to s, s_{l}}^{N,\text{pair}}$ if $s'\ne s$ and $-\sum_{\tilde{s}}r_{k,l, s \to \tilde{s}, s_{l}}^{N,\text{pair}}$ if $s=s'$.

\subsection{Representation of \texorpdfstring{$\bXN$, $G^N$ and $\bFN$}{XN, GN and FGN} as Functions from \texorpdfstring{$(0,1]$}{(0,1]}} \label{sec:representation_X_G_F_L_2}

To study the limit as $N$ goes to infinity, it will be convenient to view the functions $\bXN$ not as a vector with $N$ components $(\XN_{k})_{k\in[N]}$ but as a step function $(X^{N}_u)_{u\in(0,1]}$, where for any state $s\in\calS$, we set 
\begin{align*}
    X^{G^N}_{u,s} := X^{G^N}_{k,s}\in \{0,1\} \text{ for } u \in ((k-1)/N, k/N] .
\end{align*}
By abuse of notation, we do not introduce separate notations for the discrete and continuous variables but make the distinction by reserving the subscript letters $k,l\in[N]$ for the discrete case and $u,v\in(0,1]$ for the continuous variable. 

Similarly, we also write $G^N_{uv}=G^N_{kl}$ for $u\in((k-1)/N, k/N]$ and $v\in((l-1)/N, l/N]$, $\FN_{u,s}=\FN_{k,s}$, and $\bR^{N,\mathrm{uni}}_{k,s}=\bR^{N,\mathrm{uni}}_{u,s}$.  By using this notation, the sum of $k\in[N]$ --for instance in \eqref{eq:drift_N}-- can be replaced by an integral, i.e., $\FN_{k,s}(\bX) = \FN_{u,s}(\bX) = \bR^{N,\text{uni}}_{u,s}\bX_u  + \bX^T_u\int_0^1 \bR^{N, \text{pair}}_{u,v,s} \bX_v G^N_{uv}dv$ for $u\in((k-1)/N, k/N]$.


\subsection{Notations}

Throughout the paper, matrices, and vectors are written in bold letters, i.e. $\bX, \bx,\dots$, and regular letters are used to denote scalars like $X_{u,s}, r_{u,s\to s'}^\text{uni}$. 
The indices $s,s_k, s_k', s_l,...$ are reserved for the states, $k,l,...$ are reserved for particles, and $u,v$ are reserved for values in the unit interval. When we write that a quantity $h$ is of order $\mathcal{O}(1/N)$ or $h = \mathcal{O}(1/N)$ equivalently, this means that there exists a constant $C$ such that $h \leq \frac{C}{N}$.
Most of our results will be expressed in terms of $L_2$ norm (see the definition in Section~\ref{ssec:graphon_distance}). The space $L_2(0,1]$ refers to the quotient space of the square Lebesgue integrable functions. Throughout we will deal with vectors of $L_2(0,1]$ functions, i.e., $f, g \in L_2(0,1]^\calS$. We for the vectors $g,f$ we denote the scalar product  by $\langle f,g \rangle_{L_2(0,1]} = \sum_{s\in\calS} \int_0^1 f_{u,s} g_{u,s} \ du$ for and induced norm $\norm{f}_{L_2} = \sqrt{\sum_{s\in\calS} \int_0^1 f_u^2 du}$. For a function $G:(0,1]^2\to\R$ denote by $\opnorm{G}_{L_2} = \sup_{\{f\in L_2(0,1], \norm{f}_{L_2(0,1]} \leq 1\}} \sqrt{ \int_0^1 (\int_0^1G_{uv} f_v \ dv)^2 du}$ the $L_2$ operator norm.

\section{Limiting Graph and Graphon} \label{sec:connectivity_function}

In this section, we specify the properties that the interaction graph $G^N$ needs to satisfy as $N$ goes to infinity. To do so, we introduce the notion of graphon and define the associated cut-norm. We also introduce two sampling methods that can be used to generate a graph with $N$ nodes from a graphon. This section only reviews the material that is necessary for our results, and we refer to the famous book \cite{lovaszLargeNetworksGraph2012} for a detailed introduction to graphons.

\subsection{Graphons} \label{sec:graphons}

In this paper, what we call a graphon is a measurable function\footnote{In the literature, graphons are often restricted to symmetric functions, which correspond to undirected graphs. This restriction is not needed for our case.} $G:(0,1]^2 \to (0,1]$. The notion of graphon can be viewed as a generalization of the notion of graph. Indeed, for any $N$, a weighted graph $G^N$ can be viewed as a piecewise constant function defined on $(0,1]^2$, where the value of this function at a point $(u,v)\in(0,1]^2$ is equal to $G^N_{kl}$ whenever $u\in((k-1)/N, k/N]$ and $v\in((l-1)/N, l/N]$. Hence, a finite graph is a graphon that has a special structure. The notion of graphon generalizes the notion of finite graph by allowing $G$ to be any measurable function. We provide an illustration of a graphon and of a finite graph viewed as a graphon on Figure~\ref{fig:graphon_discretization_example}.

Throughout the paper, we consider piecewise Lipschitz continuous graphons, which are defined as follows.
\begin{definition}[Piecewise Lipschitz Graphon] \label{def:piecewise_lipschitz_graphon}
A graphon $G$ is called piecewise Lipschitz if there exists a constant $L_G$ and a finite partition of $(0,1]$ of non-overlapping intervals $\mathcal{A}_k = (a_{k-1},a_k]$ with $0= a_0< a_1 < ... < a_{K_G} = 0$ for a finite $K_G\in \N$, such that for any $k_1, k_2 \in [K+1]$ and pairs $(u,v),(u',v') \in \mathcal{A}_{k_1} \times \mathcal{A}_{k_2}$ 
\begin{align*}
\abs{G_{uv} - G_{u'v'}} \leq L_G (\abs{u - u'} + \abs{v-v'}).
\end{align*}
\end{definition}
A particular case of a piecewise Lipschitz continuous graphon is the case of a step graphon, that is such that the Lipschitz constant $L_P=0$. This implies that the function $G$ is constant on all intervals $\mathcal{A}_i \times \mathcal{A}_j$. For instance, all finite graphs can be seen as step graphons by using the partition such that $\mathcal{A}_k=(k-1/N, k/N]$ with $N$ being the number of nodes in the graph.

\begin{figure}[htb]
    \centering
    \title{Stochastic Graph Sampling, Graphon discretization, Graphon Illustrations}
    \begin{subfigure}[t]{0.25\textwidth}
        \includegraphics[width=\linewidth]{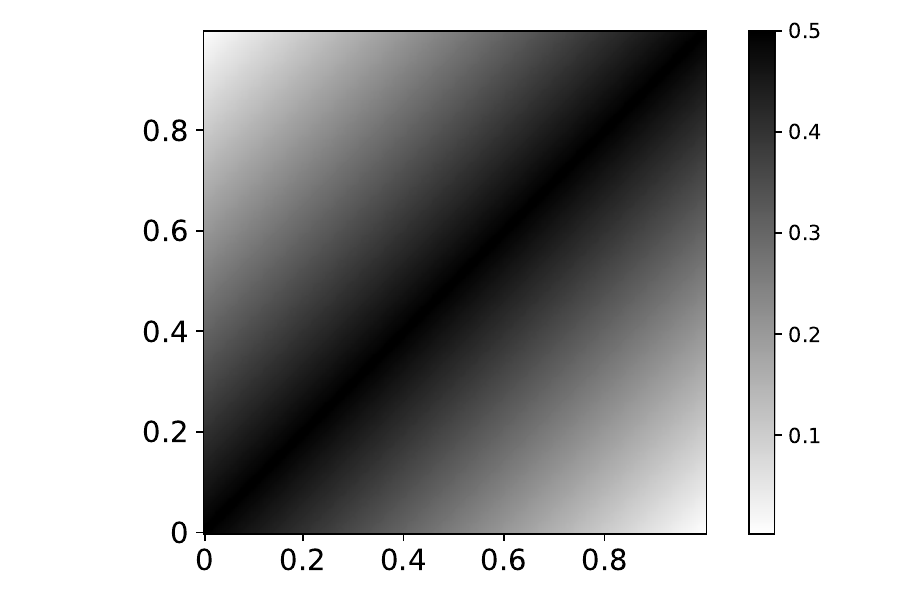}
        \caption{Graphon corresponding to $G_{uv}=\frac{1}{2}(1 - \abs{u-v})$ } \label{fig:lb_graphon_illustration}
    \end{subfigure}%
    \qquad
    \begin{subfigure}[t]{0.25\textwidth}
        \includegraphics[width=\linewidth]{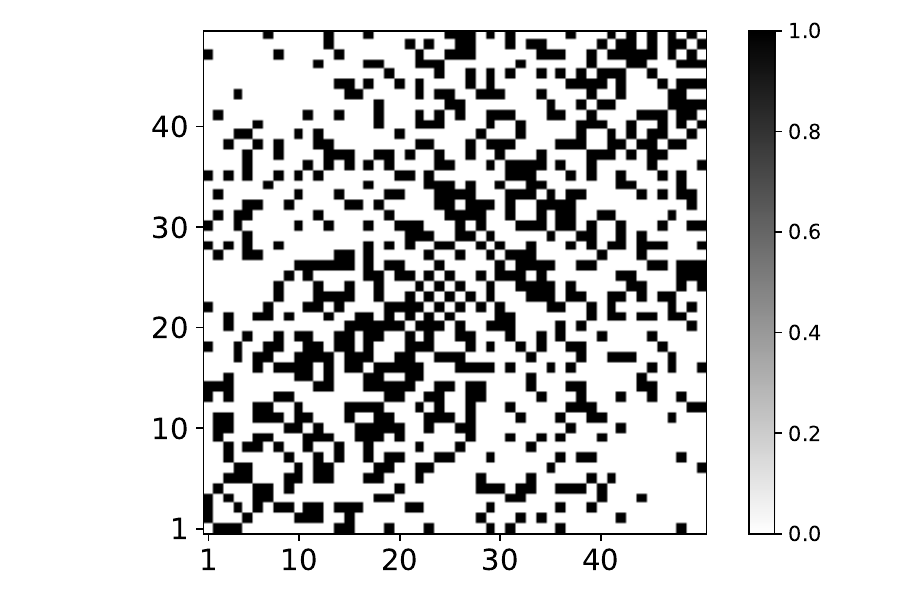}
        \caption{Step graphon associated to a sampled graph with $N=50$ nodes.}
    \end{subfigure}
    \qquad
    \begin{subfigure}[t]{0.25\textwidth}
        \includegraphics[width=\linewidth]{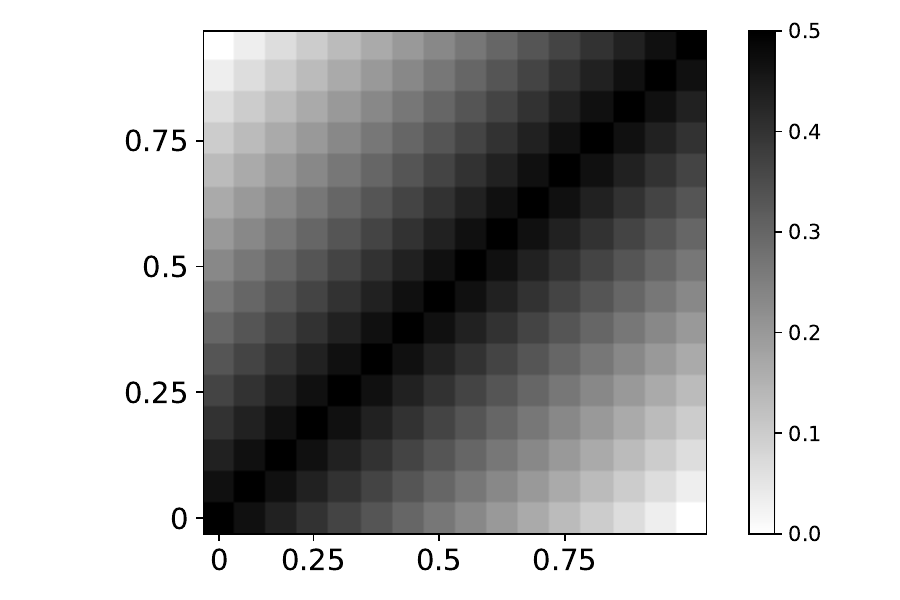}
        \caption{Step graphon obtained from graphon discretization for $N=16$.}
    \end{subfigure}\hfill
    \caption{Exemplary connectivity functions sampled after the methods of Section \ref{definition:graph_sampling} and their corresponding graphon.}
    \label{fig:graphon_discretization_example}
\end{figure}


\subsection{Generation of a Finite Graph \texorpdfstring{$G^N$}{GN} from a Graphon \texorpdfstring{$G$}{G}}
\label{definition:graph_sampling}
Based on a given graphon $G$, we consider two distinct sampling methods to generate a finite graph $G^N$ from $G$:
\begin{itemize}
    \item In the first case, termed \emph{deterministic sampling}, $G^N$ is a discretization of $G$ on $N^2$ uniformly sampled points, i.e., $G^N_{kl} = G(k/N,l/N)$ for $k,l\in[N]$.
    \item For the second case, termed \emph{stochastic sampling}, and under the preliminary assumption that the graphon is symmetric, the values $G^N$ are drawn according to independent Bernoulli random variables, i.e.,
        $G^N_{kl} = \mathrm{Bernoulli}\bigl(G(k/N,l/N)\bigr)$, and $G^N_{lk}=G^N_{kl}$ for all $k<l$.
\end{itemize}
In the case that the graphon is constant for all $u,v\in(0,1]$, the stochastic sampling method is equivalent to sampling an \ER graph. If the graphon is block-wise constant, the sampling is similar to the stochastic block model; see \cite{vanderhofstadRandomGraphsComplex2017}. 

\subsection{Graphon Distances and Convergence}
\label{ssec:graphon_distance}

To measure the distance between two graphs (and in particular to quantify how fast does a finite graph $G^N$ converge to $G$), we will use the $L_2$ operator norm. More precisely, for a measurable function $f:(0,1]\to\R$, we denote by $\norm{f}_{L_2(0,1]} = \sqrt{\int_0^1 f_u^2 du}$ the $L_2$ norm of $f$. For a graphon $G$, we denote by $\opnorm{G}$ the operator norm of $G$ in $L_2$, that is:
\begin{align*}
    \opnorm{G}_{L_2} = \sup_{\{f\in L_2(0,1] \text{ such that } \norm{f}_{L_2(0,1]} \leq 1\}} \norm{ Gf}_{L_2}.
\end{align*}
The distance between two graphons (or finite graphs) $G$ and $G'$ is measured as $\opnorm{G-G'}$.

To bound the $L_2$ distance between a piecewise Lipschitz graphon and the function associated with a randomly sampled graph, we utilize the results displayed in \cite{avella-medinaCentralityMeasuresGraphons2018}. The authors show that for a large enough (see Definition \ref{def:large_enough_N}) number $N$ of graph nodes, in our case particles, the distance between the graphon and the graph function can be upper bounded as follows:

\begin{lemma}[Theorem 1 from \cite{avella-medinaCentralityMeasuresGraphons2018}] \label{lemma:graphon_distance}
 Let $G$ be a symmetric piecewise Lipschitz graphon, and let $G^N$ be a stochastic sampling of it as defined in \ref{definition:graph_sampling}. Then, for large enough $N$ (as in Definition~\ref{def:large_enough_N}) with probability as least $1 - \delta$ the distance in the $L_2$ operator norm between the graphon $G$ and the sampled graph $G^N$ is bounded by
\begin{align}
\opnorm{G - G^N}_{L_2} \leq \sqrt{\frac{4 \log{(2N/\delta)}}{N}} + 2\sqrt{\frac{(L^2_G - K^2_G)}{N^2} +\frac{K_G}{N}} =: \psi_{\delta, G}(N). \label{eq:graphon_sampled_graph_distance}
\end{align}
\end{lemma}

\begin{definition}[Large Enough $N$ \cite{avella-medinaCentralityMeasuresGraphons2018}] \label{def:large_enough_N} 
Given a piecewise Lipschitz graphon $G$ with partition of $(0,1]$ of non-overlapping intervals $\mathcal{A}_k = (a_{k-1},a_k]$ as in Definition \ref{def:piecewise_lipschitz_graphon} and $\delta < e^{-1}$. The quantity $N$ is called `large enough', if 
\begin{subequations}
\begin{align}
& \frac{2}{N} < \min_{k\in \{0..K_G\}} (a_{k} - a_{k-1}), \\
& \frac1N\log\left(\frac{2N}{\delta}\right) + \frac{2K_G + 3L_G}{N} < \sup_{u\in(0,1]} \int_0^1 G_{u,v}\ dv \quad \text{and} \\
& Ne^{-N/5} < \delta.
\end{align}
\end{subequations}
\end{definition}

Implied by Theorem \ref{lemma:graphon_distance}, by setting $\delta_N = \frac{2}{N}$, we see that the distance between a piecewise Lipschitz graphon and a sampled graph is with high probability of order $\opnorm{G - G^N}_{L_2(0,1]} = \mathcal{O}\left(\sqrt{\frac{\log(N)}{N}}\right)$. The meaning of `with high probability' in this context is that the right-hand side holds with probability at least $1 - 2/N$.

\section{Main Results} \label{sec:main_results}


\subsection{Assumptions}


In Section~\ref{sec:particle_system}, we constructed a model that depends on a scaling parameter $N$. We state the necessary assumptions that will be used to state the accuracy of the graphon mean field approximation. 
\begin{enumerate}[label=($A_\text{\arabic*}$)]
    \item The state space $\calS$ is finite. \label{assumption:finite_states}
    \item The graphon $G$ is piecewise Lipschitz continuous. \label{assumption:piecewise_lipschitz_graphon}
    \item There exists bounded and piecewise Lipschitz-continuous rate functions for $r_{u,s_u\to s_u'}^{\text{uni}}$ and $r_{u,v,s_u\to s_u', s_{v}}^{\text{pair}}$ for $u,v \in (0,1]$ and $s_u,s_u',s_v \in \calS$ such that the rates function of the original $N$ particle systems have the relation:
    \begin{align}
        r_{k/N,s_k\to s_k'}^{\text{uni}} = r_{k, s_k \to s_{k}'}^{N, \text{uni}} \quad \text{ and } \quad r_{k/N, l/N,s_k\to s_k', s_l}^{\text{pair}}= r_{k, l, s_k \to s_{k}', s_l}^{N, \text{pair}} \quad \text { for } k, l\in [N].
        \label{eq:rate_N}
    \end{align}
    \label{assumption:bound_lipschitz_transitions}
\end{enumerate}
The first assumption is technical and simplifies the definition of the $L_2$ space. The second assumption is very classical regularity assumption when studying sequences of graphs that converge to graphons \cite{lovaszLargeNetworksGraph2012}. In practice, our bounds will depend on the distance between the original graph $G^N$ and the graphon $G$. The last assumption ensures that the particle transition rates for the finite and the graphon system are equal. We point out that assumption \ref{assumption:bound_lipschitz_transitions} can be generalized by replacing the equality in \eqref{eq:rate_N} with bounds for the distance of $r^{N,\cdot}$ and $r^{\cdot}$ therefore modifying the bounds of the theorem to include terms of the form $\norm{r^{N,\cdot}-r^{\cdot}}$.


\subsection{The Graphon Mean Field Approximation} \label{sec:limiting_system}

The graphon mean field approximation aims at approximating the dynamics of the original system $\bX$. We define the graphon related drift 
similarly to the drift of the particle system in Equation \eqref{eq:drift_N}. For $u,v\in (0,1]$ the graphon based drift is defined by 
\begin{align}
    \FG_{u,s}(\bx) = \bR^{\text{uni}}_{u,s}\bx_u + \bx_u^T \int_0^1 \bR^{\text{pair}}_{u,v,s} G_{u,v}\bx_v dv. \label{eq:ode_drift_definition}
\end{align}
This equation is identical to the original drift equation \eqref{eq:drift_N} with two modifications: First, the rates do not depend on $N$. Second, the discrete variables $k,l\in[N]$ are replaced by continuous variables  $u,v\in(0,1]$, which notably replaces the sum over $l$ by an integral over $v$.
    
Based on the drift function $\bFG$ and the initial condition $\bx_0$, we call the graphon mean field approximation the solution of the differential equation 
\begin{align}
    \bODE(t, \bx_0) = \bx_0 + \int_0^t \bFG(\bODE(\tau,\bx_0))d\tau. \label{eq:definition_ode}
\end{align}

\begin{lemma} \label{lemma:unique_solution}
    Let $\bFG$ denote the deterministic drift defined in Equation \eqref{eq:ode_drift_definition} and $\bODE(t, \bx_0)$ the solution of the graphon mean field approximation at time $t$ with initial condition $\bx_0$ as in Equation \eqref{eq:definition_ode}. It holds that $\bODE$ is well-defined, has a unique solution and is differentiable with respect to its initial condition. Furthermore, this derivative is Lipschitz continuous.
\end{lemma}

The proof is postponed to Section~\ref{proof:unique_solution}.

\subsection{Accuracy of the Approximation}


The following result provides a bound on the distance between the stochastic system and the graphon mean field approximation. This bound is stated as the $L_2$ distance between $\esp{\bXGN}$ and $\bODE$. Recall that by definition, $\E[\XGN_{k,s}(t)] = \mathbb{P}(S_k(t) = s)$ is the probability for an item $k$ to be in state $s$ at time $t$. Hence, our theorem shows that the graphon mean field model is an accurate approximation of the state distribution of the particles. The statement of the theorem should be interpreted as saying that the distribution of the particles over the states $\calS$ for any time $t\geq0$ is approximated by the graphon mean field approximation $\bODE$ with accuracy depending on the system size $N$ as well as the distance between $G^N$ and graphon $G$. In particular, if $G^N$ converges to $G$ (for the graphon norm), then the graphon mean field approximation is asymptotically exact.



\begin{theorem}[$L_2$ convergence] \label{thrm:approximation_accuracy}
Let $\bXN(t) = (\XN_{k,s})(t)_{k\in[N], s\in \calS}$ be the stochastic particle system of size $N$ related to a graph instance $G^N$. Let $\bODE (t) = (\ODE_{u,s}(t,\bx_0) )_{u \in(0,1], s \in \calS}$ be the mean field approximation of the particle system as defined in Equation \eqref{eq:definition_ode} for an initial condition $\bx_0 = \bXN(0)$\footnote{Here, $\bXN(0)$ is interpreted as a vector in $L_2(0,1]^{\abs{\calS}}$.}. Assume additionally that conditions \ref{assumption:finite_states} - \ref{assumption:bound_lipschitz_transitions} are fulfilled and let $t>0$ be arbitrary but fixed. Then, there exist constants $C_A, C_B \geq 0$ such that
\begin{align}
\norm{ \E[\bXN(t) \mid \bXN(0), G^N ] - \bODE(t,\bXN(0)) }_{L_2} \leq \frac{C_A}{N} + C_B \opnorm{G^N -G}. \label{eq:thrm_approximation_accuracy}
\end{align}
\end{theorem}

The proof of Theorem \ref{thrm:approximation_accuracy} is postponed to Section \ref{proof:approximation_accuracy}.    
The constants $C_A, C_B$ of Equation \eqref{eq:thrm_approximation_accuracy} depend on the uniform bound of the rates, the time $t$ and the Lipschitz constants of the graphon.  In the subsequent Corollaries \ref{corollary:het_approx_results} and \ref{corollary:graphon_results} we will see, that if the $G^N$ is generated by one of the methods illustrated in Section \ref{definition:graph_sampling}, precise bounds on the distance $\opnorm{G^N -G}$ can be obtained. To illustrate our results and underline that the constants are small in practice, we provide examples in Section \ref{sec:numerical_experiments}. 
We point out that Theorem \ref{thrm:approximation_accuracy} is applicable for a wide range of construction methods for $G^N$. The subsequent corollaries illustrate cases of stochastic and deterministic sampling methods.  We emphasize, however, that any method that allows the construction of densely connected graphs, for which bounds of $\opnorm{G^N -G}$ are attainable, is viable. At last, we want to point out that by using the same framework, it seems feasible to extend our results to interactions of triplets or higher order interactions seems feasible. Yet, due to our generic choice of transition rates, this would be linked to increasingly heavy notations for the dynamics while only giving little more insight into the accuracy of the graphon mean field method. 

\subsection{Case Specific Bounds for \texorpdfstring{$\opnorm{G^N- G}_{L_2}$}{|||GN - G|||}}

The subsequent corollaries, give specific bounds for the case that $G^N$ was generated as described in Section \ref{definition:graph_sampling}. In the first case, if $G^N$ is obtained though discretization of $G$, Corollary \ref{corollary:het_approx_results} shows that the accuracy is of order $\mathcal{O}(1/N)$. Our second Corollary \ref{corollary:graphon_results} specifies big-$\mathcal{O}$ convergence rates if the graph $G^N$ is sampled stochastically from the graphon. In this case, the accuracy is with high probability of order $\mathcal{O}(\sqrt{\frac{\log{N}}{N}})$.

\subsubsection{Case 1: Graphon Discretization}

The corollary give accuracy bounds for the approximation in the case that the graph $G^N$ is obtained as a discretized version of $G$. The result gives bounds on the difference between the distributions of the particles in the stochastic system and the approximate values obtained through the graphon mean field approximation. 

\begin{corollary}[Dense Heterogeneous System] \label{corollary:het_approx_results}
Assume \ref{assumption:finite_states} - \ref{assumption:bound_lipschitz_transitions} and let $\bODE$ and $\bXN$ be defined as in Theorem \ref{thrm:approximation_accuracy}. Let $k \in[N], s\in\calS$ and $t\geq0$ be arbitrary but fixed. If $G^N$ is generated by the deterministic sampling method of \ref{definition:graph_sampling}, i.e., a discretization of the graphon $G$, it holds that
\begin{align}
\norm{ \E[\bXN(t) \mid \bXN(0), G^N ] - \bODE(t,\bXN(0)) }_{L_2} & \le \frac{C_A+C_B C_{G^N}}{N}
\end{align}
\end{corollary}

The proof is postponed to Section \ref{proof:het_approx_results}. The constants $C_A, C_B$ are as in Theorem \ref{thrm:approximation_accuracy}. The additional constant $C_{G^N}$ relates to the discretization error of the deterministic sampling method. For this case, it is noteworthy that the accuracy of the approximation aligns with the results one obtains for the homogeneous setting as described in \cite{gastExpectedValuesEstimated2017,gastRefinedMeanField2017} while allowing for heterogeneous connectivity and rates among the population. 


\subsubsection{Case 2: Random Graph}

Our second corollary provides accuracy bounds for interacting particle systems on dense random graphs. By the definition of the graph sampling methods, see Section \ref{definition:graph_sampling}, with high probability, the number of connected neighbors for each particle is of order $N$. This ensures that for large enough systems that the neighborhood of each node serves as a local representation of the overall system state. This leads to the following result: 

\begin{corollary}[Graphon System Approximation] \label{corollary:graphon_results}
Assume the conditions \ref{assumption:finite_states} - \ref{assumption:bound_lipschitz_transitions} as in Theorem \ref{thrm:approximation_accuracy} for a symmetric graphon $G$. Let $\bXN(t) = (\XN_{k,s})(t)_{k\in[N], s\in \calS}$ be a stochastic particle system of size $N$ with $G^N$ obtained through the stochastic sampling method as defined in \ref{definition:graph_sampling}. Let $t\geq 0$ and $k\in[N], s\in \calS$ be arbitrary but fixed. Then, with probability at least $1-2/N$ and for `large enough' $N$, as defined in Definition \ref{def:large_enough_N}, the graph $G^N$ is sampled such that
\begin{align}
\norm{ \E[\bXN(t) \mid \bXN(0), G^N ] - \bODE(t,\bXN(0)) }_{L_2} \leq \frac{C_A}{N} + C_B \psi_{G}(N)
\end{align}
where 
$\psi_{G}(N) := \sqrt{\frac{8 \log{(N)}}{N}} + 2\sqrt{\frac{(L_G^2 - K_G^2)}{N^2} +\frac{K_G}{N}}$
with $L_G$ being the Lipschitz constant of the graphon $G$ and $K_G$ the size of the partition as defined in \ref{def:piecewise_lipschitz_graphon}. 
\end{corollary}

The proof of the corollary is postponed to Section~\ref{proof:graphon_results}.

 
\section{Numerical Experiments} \label{sec:numerical_experiments}

In this section, we present two examples which support the statements of our theoretical results and illustrate the applicability of the framework. For the first example, we look at a load balancing model with communication restrictions of the servers imposed by a graph. In this example, the focus is on the stochastically sampled graph, which imposes heterogeneous rates due to connectedness of the servers. For the dynamics of the system, we see each node as a dispatcher / server pair applying the JSQ($2$) policy with respect to itself and connected servers whenever a job arrives. For the second example, we consider a heterogeneous bike-sharing system. Here, the focus lies on the heterogeneity of the popularity of the stations, which affects the flow of bikes through the system.


\subsection{Load Balancing Example} \label{subsec:load_balancing_example}

\subsubsection{Model}

The considered load balancing model is a jump process in the Markovian setting. We consider a system where the servers-dispatcher pairs are connected through a graph structure where each server-dispatcher is represented by a node. All servers have a finite maximal queue length $K_L\in\N$. The connections between the server-dispatcher pairs are denoted by $G^N_{kl}$ and are sampled according to Section \ref{definition:graph_sampling} using the stochastic sampling method. The graphon $G$ used to sample the edges is the same as Figure \ref{fig:lb_graphon_illustration}. Based on the described connectivity structure, jobs arrive to a server-dispatch pair following a Poisson arrival process with rate $\lambda_L > 0$. Arriving jobs are distributed according to the JSQ($2$) policy idea, i.e., the dispatcher considers its own server state as well as another randomly sampled but connected server and forwards the job to the server with the lesser load. In the case that both servers have the same queue length, the job is assigned randomly among the two. In case both servers have a full queue, the job is discarded. The service time of a job is exponentially distributed with mean $\mu_L$ and jobs are handled in first come, first served order. For a system of size $N$ with graph instance $G^N$ and state $\bXN = (\XN_{k,s})_{k\in[N], s=0..K_L}$ the transitions of the system are
\begin{align}
& \bXN \to \bXN + \BFe_{k,s+1}^N - \BFe_{k,s}^N \\
& \quad \text{ at rate } \lambda_L \XN_{k,s}\sum_{l\in[N]} (\frac{G^N_{kl}}{d\toN(k)} + \frac{G^N_{kl}}{d\toN(l)}) \Bigl(  \frac{1}{2} X_{l,s} + \sum_{s_l \ge s+1} X_{l,s_l}  \Bigr)\mathds{1}_{s < K_L} \label{graphon:eq:job_arrival} \\
& \bXN \to \bXN + \BFe_{k,s-1}^N - \BFe_{k,s}^N \quad \text{ at rate } \XN_{k,s} \mu_L \mathds{1}_{s>0} \label{graphon:eq:job_completion}
\end{align} 
where $\BFe_{k,s}$ is a $N\times \abs{\calS}$ matrix having a one at the $(k,s)$ entry and zero values everywhere else and $d\toN(k)$ is the degree of node $k$. Equation \eqref{graphon:eq:job_completion} corresponds to the completion of a job by server~$k$ when the queue is of size $1 \leq s \leq K_L$. The second type of transitions, equation \eqref{graphon:eq:job_arrival}, corresponds to adding a job to  $k$ having $0 \leq s \leq K_L-1$ jobs in the buffer. In this case, the queue size is increased by one from $s$ to $s+1$. To explain the transition rate we see that the servers~$k$ can be selected in two ways. By selecting $k$ or another connected server $l$ first and the other second. In the case that both queues have equal length, the chance that the job to arrives at server~$k$ is $1/2$. If both buffers are full, the job is discarded. In the case that a node associated to server $l$ is isolated, i.e., has no edges, we define $\frac{G^N_{kl}}{d\toN(l)}$ to be zero.

\subsubsection{Drift and Graphon Mean Field Approximation} 
For the deterministic drif, we replace the values of $\frac{G^N}{d\toN(k)}$ by the ones of the graphon $\frac{G}{d(v)}$ with $d(v) := \int_0^1 G_{v,\nu} d\nu$ and the sums over particles by an integral over $(0,1]$. For a given state $\bx = (x_s)_{s=0,\dots,K_L}$ with $x_s \in L_2(0,1]$, the drift evaluated at $(u,s) \in(0,1]\times \{0,\dots,K_L\}$ is defined by
\begin{align*}
\FG_{u,s}(\bx) = x_{u,s} \mathds{1}_{s < K_L} \int_0^1 \lambda_L \left(\frac{G_{u,v}}{d(u)} + \frac{G_{u,v}}{d(v)}\right) \bigl( \frac{1}{2}x_{v,s} + \sum_{s_v\ge s+1} x_{v,s_v} \bigr)dv \  - x_{u,s} \mu_L \mathds{1}_{s>0}.
\end{align*}
The graphon mean field approximation is then defined as in Equation \eqref{eq:definition_ode}.

\subsubsection{Derivation of Accuracy Bounds}

While edges between the server-dispatcher pair are sampled according to the stochastic sampling method of Section \ref{definition:graph_sampling}, the dependence on of the rates on the node degree, prevents a direct application of the results of Corollary \ref{corollary:graphon_results}. In particular, to apply Corrollary \ref{corollary:graphon_results} it is assumed that the graph edges are scaled by a factor of $1/N$ instead of the node degree. To resolve this issue and obtain accuracy bounds similar to the one given by the Corollary, we start by defining 
$\tilde{G}^N_{kl} := (\frac{G^N_{kl}}{d\toN(k)} + \frac{G^N_{kl}}{d\toN(l)})\frac{N}{16}$, $r^{N,\text{pair}}_{k,l,s\to s+1,s_l} = \lambda_L ( \frac{1}{2} \mathds{1}_{s_l = s} + \mathds{1}_{s_l > s}) \mathds{1}_{s < K_L}$, and
$r^{N,\text{uni}}_{k,s\to s-1} = \mu_L \mathds{1}_{s>0}$
with $r^{N,\text{pair}}$ being the pairwise transitions and $r^{N,\text{uni}}$ the uniateral transitions for $s, s', s_l \in \{0, ..., K_L \}$. This allows to obtain the drift of the stochastic system similar to the one introduced in Equation \eqref{eq:drift_N} namely
\begin{align*}
\FN_{k,s}(\bX) =  X_{k,s} \sum_{s_l \ge s} \sum_{l\in[N]} 16 \frac{\tilde{G}^N_{kl}}{N} r^{N,\text{pair}}_{k,l,s\to s+1,s_l} X_{l,s'} - X_{k,s} r^{N,\text{uni}}_{k,s\to s-1}.
\end{align*}
We proceed similarly for the drift of the graphon mean field approximation by defining
$\tilde{G}_{uv} := (\frac{G_{u,v}}{d(u)} + \frac{G_{u,v}}{d(v)})\frac{1}{16}$, $r^{\text{pair}}_{k,l,s\to s+1,s_l} :=  \lambda_L ( \frac{1}{2} \mathds{1}_{s_l = s} + \mathds{1}_{s_l > s}) \mathds{1}_{s < K_L}$ and $r^{N,\text{uni}}_{k,s\to s-1} := \mu_L \mathds{1}_{s>0}$ to rewrite 
\begin{align*}
\FG_{u,s}(\bx) = x_{u,s}  \int_0^1 \sum_{s_v \geq s} 16 \tilde{G}_{uv}  r^{\text{pair}}_{k,l,s\to s+1,s_v} x_{v,s_v} dv  - x_{u,s} r^{N,\text{uni}}_{k,s\to s-1}.
\end{align*}
Note that for the above reformulations, the result of our main Theorem \ref{thrm:approximation_accuracy} is still applicable and yields constants $C_A, C_B\geq 0$ such that  $\norm{\E[\bXN(t)\mid G^N] - \bODE(t)}_{L_2} \leq \frac{C_A}{N} + C_B \opnorm{\tilde{G}^N - \tilde{G}}_{L_2}$. Using the definition of the $\tilde{G}^N, \tilde{G}$, we obtain $\opnorm{\tilde{G}^N - \tilde{G}}_{L_2} \leq \frac{1}{16}\Bigl(\norm{G^N - G}_{L_2}  \frac{N}{\min_k d^N(k)} + \sup_{u\in(0,1]} \abs{ \sum_{k=1}^N\frac{N}{d^N(k)}\mathbf{1}_{u\in(k-1/N,k/N]} - \frac{1}{d(u)}} \Bigr)$. For $\opnorm{G^N - G}_{L_2}$, we can use the bound as in Corollary \ref{corollary:graphon_results} which is of order $\mathcal{O}(\sqrt{\frac{\log N}{N}})$ with probability at least $1-\frac{2}{N}$ and large enough $N$ (as in Definition \ref{def:large_enough_N}). 
To obtain similar bounds for the node degree, the multiplicative Chernoff bound can be used, i.e., 
$\Prob\Bigl(\frac{N}{d^N(k)} \geq \frac{N}{(1 - \gamma)\E[d^N(k)]}\Bigr) = \Prob \Bigl(d^N(k) \leq (1-\gamma)\E[d^N(k)] \Bigr) \leq \exp(- \frac{\E[d^n(k)]\gamma^2}{3})$. Taking $\gamma = \sqrt{\frac{3\log{N}}{\E[d^N(k)]}}$ yields $\Prob\Bigl(\frac{N}{d^N(k)} \geq \frac{N}{(1 - \gamma)\E[d^N(k)]}\Bigr)\leq \frac{2}{N}$.
By the triangle inequality we have $\abs{ \sum_{k=1}^N\frac{N}{d^N(k)}\mathbf{1}_{u\in(k-1/N,k/N]} - \frac{1}{d(u)}} \leq \abs{ \sum_{k=1}^N(\frac{N}{d^N(k)} - \E[\frac{N}{d^N(k)}])\mathbf{1}_{u\in(k-1/N,k/N]}} + \abs{\sum_{k=1}^N \E[\frac{N}{d^N(k)}]\mathbf{1}_{u\in(k-1/N,k/N]} - \frac{1}{d(u)}}$. We apply Chernoff bound the first summand to show that the difference concentrates around zero, i.e., $\Prob\Bigl(\frac{N}{d^N(k)} \leq \frac{N}{(1 - \gamma)\E[d^N(k)]}\Bigr) \geq 1 - 2/N$ and therefore $\frac{N}{d^N(k)} - \frac{N}{\E[d^N(k)]} \leq N ( \frac{1}{(1-\gamma)\E[d^N(k)]} -  \frac{1}{\E[d^N(k)]}) 
= \mathcal{O}(\sqrt{\frac{\log{N}}{N}})$ with probability at least $1 - 2/N$. The case $\frac{N}{d^N(k)} \leq \frac{N}{(1 - \gamma)\E[d^N(k)]}$ follows using the same arguments. For the second summand we use the statement of Lemma \ref{lemma:lipschtiz_discretization_error} to obtain an error bound of order $\mathcal{O}(1/N)$. This shows that for large enough $N$, $\opnorm{\tilde{G}^N - \tilde{G}}_{L_2}$ is of order $\mathcal{O}(\sqrt{\frac{\log N}{N}})$.


\subsubsection{Implementation}

For the simulation, we set the parameters of the load balancing system as follows:
The arrival rate is set to $\lambda_L = 1$, the server rate to $\mu_L = 1.1$ for all servers, and the maximal server capacity is $K_L = 10$. For each system size $N$, we sample a graph that remains fixed for all simulations and obtain the sample mean by averaging over $2000$ sampled trajectories for all system sizes. To compute the solution of the graphon mean field approximation, we discretize the drift by a simple step discretization. To be more precise, let $\gamma \in \N$ be the discretization parameter and $U^\gamma = \{ U_i^\gamma\}_{i=1..\gamma}$ be the even partition of the interval $(0,1]$ with $U_i^\gamma := (i-1/\gamma, i / \gamma]$. We define by $F^\gamma_B(\bx)_{u,s}:= \sum_{i=1}^\gamma \mathds{1}_{U_i^\gamma}(u) F_B(\bx)_{i/\gamma,s}$ the values of the discretized version of $F_B$. Throughout the numerical experiments, we set $\gamma = 100$. We choose this simple discretization method as it provides good approximation results and low computation times, i.e., for a rudimentary implementation using a NumPy ode solver, we are able to solve the discretization in a few seconds.\footnote{For intricate intensity and graphon functions it might be necessary to refine and tune the discretization approach.} The Figures \ref{fig:single_items_lbs} and \ref{fig:queue_at_least_figure} show the results of our numerical experiments. The first Figure \ref{fig:single_items_lbs} shows the approximation accuracy for a single server and different system sizes. We see that already for $N=40$ the sample mean is very close to the approximation value, which supports the statement of Corollary \ref{corollary:graphon_results}. In the second Figure \ref{fig:queue_at_least_figure}, we compare the values for the fraction of servers with at least $s$ jobs for $s=1,...,4$ as described in the figure caption. Remarkably here, even for small system sizes, the values obtained by the sample mean are close to the ones obtained by the approximation. As we observe the by Corollary \ref{corollary:graphon_results} suggested behavior for single particles inb Figure \ref{fig:single_items_lbs}, the increased accuracy for $N=20, 30$ is likely caused by an averaging effect and a well-connected graph instance $G^N$.

\begin{figure}[htbp]
  \centering
  \includegraphics[width=\linewidth]{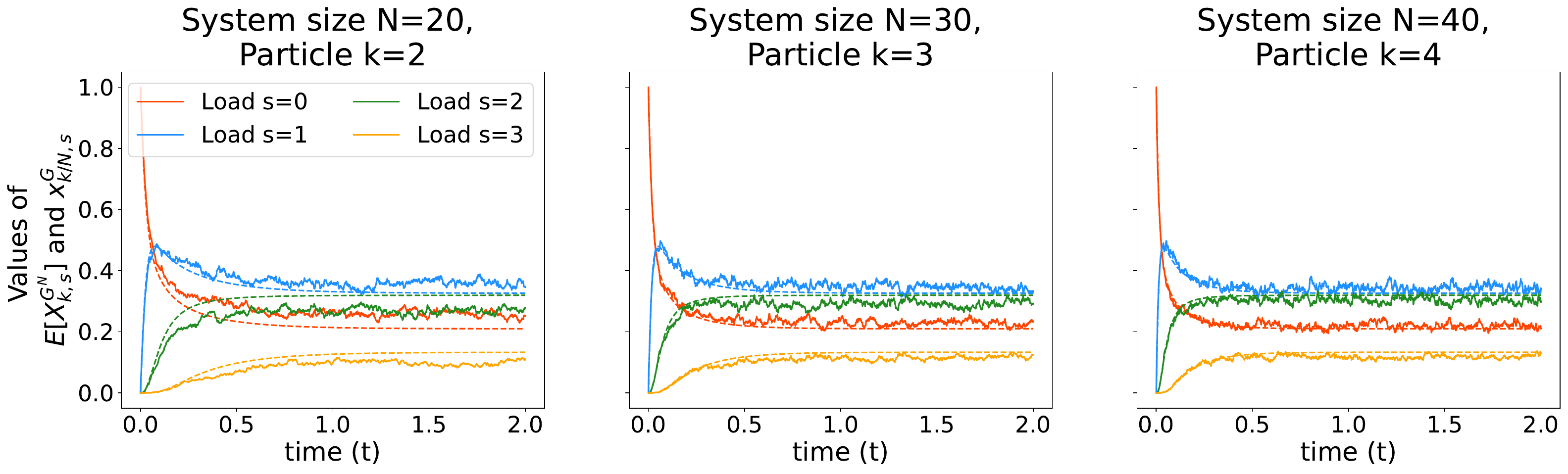}
  \caption{The figure shows the evolution of $\mathbb{P}(\XN_{k,s}(t)=1)=\E[\XN_{k,s}(t)]$, the probability for a server-dispatcher pair to have $s=0,...,3$ jobs for $t\in[0,2]$. In each plot the results for one system size $N=20,30,40$ are displayed. Throughout, the sample mean values are plotted against the values of the graphon mean field approximation.}
  \label{fig:single_items_lbs}
\end{figure}

\begin{figure}[htbp]
    \centering
    \includegraphics[width=\linewidth]{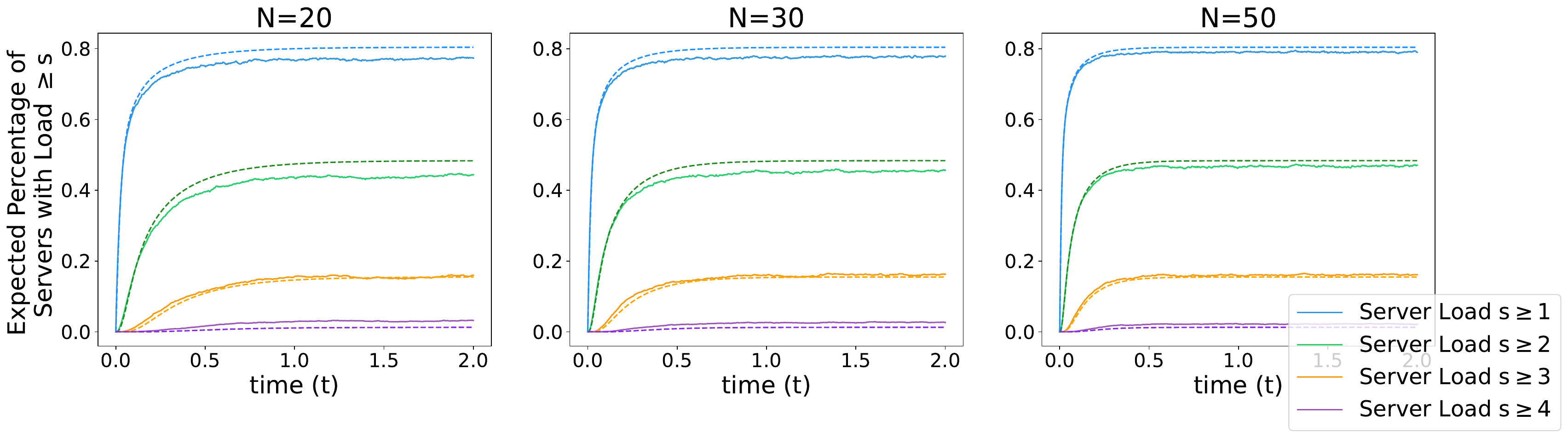}
    \title{Simulation and Approximation for the Load Balancing Example}
    \caption{The figure compares the values for the fraction of servers with at least $s$ jobs obtained by the sample mean and the approximation. The quantities are calculated as $\E[Q_s(t)] = \E [ \sum_{k\in[N]} \frac{1}{N}\sum_{s'\geq s} \XN_{k,s'}]$ for the sample mean and $q_s(t) =  \int_0^1 \sum_{s'\geq s} \ODE_{u,s'}du$ for the approximation.}
    \label{fig:queue_at_least_figure}
\end{figure}


\subsection{Heterogeneous Bike-Sharing System}

\label{sec:bike_example}

\subsubsection{Model}

We consider a bike-sharing model following the setup used in \cite{frickerIncentivesRedistributionHomogeneous2016,frickerMeanFieldAnalysis2012}. The model consists of $N\in\N$ bike stations, representing the particles in the system. Each station has finite capacity $K_B\in\N$. The system has a fleet of bikes of size $M:= \lfloor \alpha N \rfloor$ which, at time $t=0$, is evenly distributed amongst the stations, i.e., $\alpha$ bikes per station. The evolution of the system depends on the movement of bikes. We differ between the two cases, bikes in the system can either be stationary or in transit between stations. To align with the notations used in the theory part of the paper, we denote by $S_k^N(t), i\in[N]$ the number of bikes at station $i$ for time $t \geq 0$. For the system, heterogeneity comes from the varying popularity of stations. Hence, we denote by $p_B:(0,1] \mapsto \R_{\geq0}$ the popularity function. Based on the popularity function $p_B$, we define the graphon of the bike sharing system by $G_B(u,v) := p_B(v)/\int_0^1 p_B(\nu)d\nu$. For the system of size $N$, connectivity for the stations is obtained by deterministic sampling as described in Section \ref{definition:graph_sampling}, i.e., we discretize the graphon based on the system size $N$. For two stations $k,l\in[N]$ connectivity is therefore defined by $G^{N,B}_{kl}:= G_B(k/N, l/N) = p_B(l/N)/\int_0^1 p_B(\nu)d\nu$. Based on the connectivity between stations, it remains to define the dynamics of the system. For a given state $(S_1,...,S_N)\in[K_B]^N$ the bikes move between stations in the following way: 
\begin{itemize}
\item Customers arrive at a stations $k\in[N]$ with rate $\lambda_B$ leading to the transitions
\begin{align*}
S \to S - \BFe_{k}^N && \text{ at rate } && \lambda_B\mathds{1}_{S_k>0}. \phantom{\frac{p_B(k/N)}{\int_0^1 p_B(\nu)d\nu} (M - \sum_k S_k^N)}
\end{align*}
\item The travel time of bikes between stations is exponentially distributed with rate $\mu_B > 0$. Hence, the for station $k\in[N]$, the arrival rate of bikes is the product of the scaled popularity of station times the traveling bikes $(M-\sum_k S_k^N)$ weighted by the travel time, i.e., $\frac{p_B(k/N)}{\int_0^1 p_B(\nu)d\nu} \bigl( M - \sum_k S_k^N \bigr)\mu_B$. This implies the transition
\begin{align*}
S \to S + \BFe_{k}^N && \text{ at rate } && \frac{p_B(k/N)}{\int_0^1 p_B(\nu)d\nu} (M - \sum_k S_k^N)\mu_B \mathds{1}_{S_k<K_B}.
\end{align*}
\end{itemize}
For the above, the notation $\BFe_{k}^N$ refers to the unit vector of size $N$ having a one at entry $k$ and zeros everywhere else.

\subsubsection{Drift \& Limiting System}
To define the limiting system, we derive the drift implied by the above transitions and the defined graphon $G_B$. We start by considering the indicator state representation as outlined in Section \ref{sec:particle_system}, i.e., the drift is a vector of size $K+1$ of $L_2(0,1]$ functions. In contrast to the stochastic rates, for the drift, the sum over particles is replaced by an integral over and the values of $G^N$ replaced by the graphon values. For a state $\bx = (x_s)_{s=0,\dots,K}$ with $x_s \in L_2(0,1]$, the drift of the system at $(u,s) \in (0,1] \times \calS = \{0,\dots,K_B\}$ is defined by
\begin{align*}
\FG_{u,s}(\bx) := x_{u,s} \left( \frac{p_B(u)}{\int_0^1 p_B(\nu)d\nu} \bigl(M - \int_0^1 \sum_{s=0}^{K_B} x_{v,s}dv \bigr) \mu_B \mathds{1}_{s<K_B} - \lambda_B \mathds{1}_{s>0}\right).
\end{align*}
By definition of the system and particularly the graphon, it is immediate that the assumptions of Theorem \ref{thrm:approximation_accuracy} and Corollary \ref{corollary:het_approx_results} hold, therefore guaranteeing the accuracy of the approximation.

\subsubsection{Implementation \& Results}

For our simulations, we set the popularity function of the bike sharing system to $p_B(v) := 1 - 0.5v$. The travel rate and customer arrival rate are $\mu_B=1$ and $\lambda_B=1$ respectively. For the plots of Figure \ref{fig:plots_distribution_approximation_bikes_single} and Figure \ref{fig:ode_state_approximation} we calculate the sample mean by averaging over $7500$ simulations for each system size. For each plot in the first Figure \ref{fig:plots_distribution_approximation_bikes_single}, the mean field trajectory for single item and state is plotted against the sample mean. Here the horizontal axis represents time $t$ and the vertical axis the probability for the item to be in the state. In the second Figure \ref{fig:ode_state_approximation}, a comparison of the sample mean against the values of the approximation for fixed time and state is shown. As used for the theoretical results, the state of the stochastic system is represented as a step function with constant value $\E[\XN_{k,s}]$ on the intervals $(k-1/N,k/N]$ for $k=1..N$. In accordance with our theoretical results, the plot shows that with increasing system size the graphon mean field approximation becomes more accurate and captures the state distribution of the particles well. 
\begin{figure}[htbp]
    \centering
    \title{Simulation and Approximation for the Load Balancing Example}
    \begin{subfigure}[b]{\textwidth}
        \includegraphics[width=\linewidth]{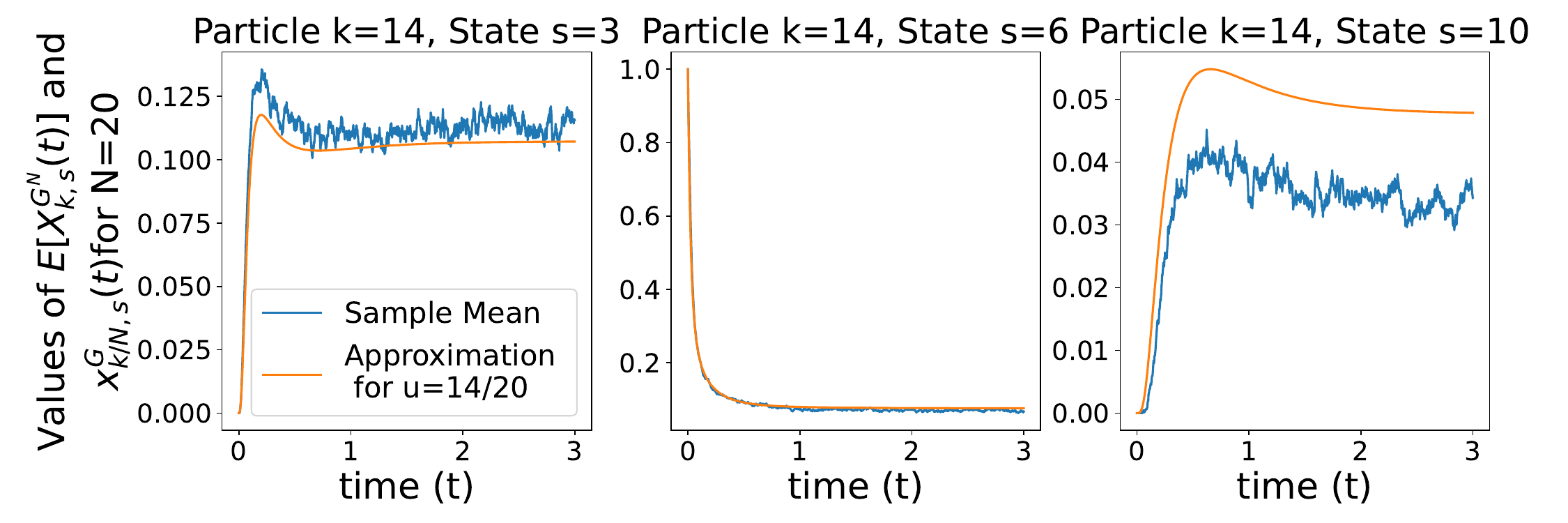}
    \end{subfigure}\\
    \begin{subfigure}[b]{\textwidth}
        \includegraphics[width=\linewidth]{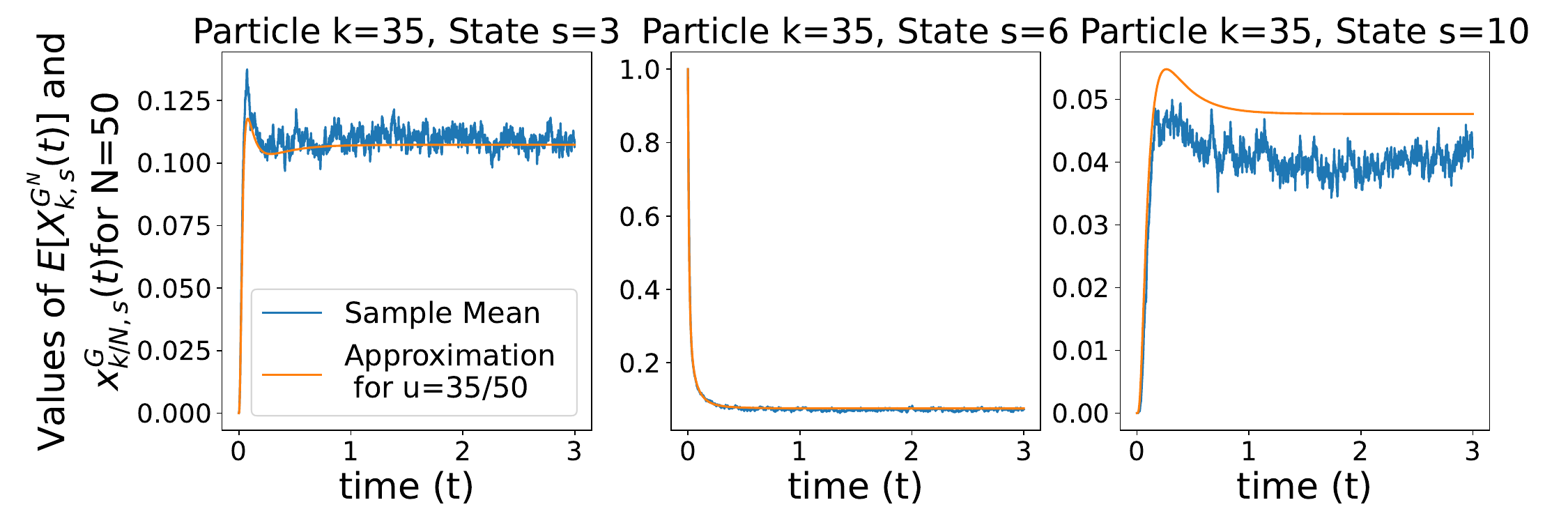}
    \end{subfigure}\\
    \caption{The plot shows the sample mean of $\E[\XN_{k,s}(t)]$ for $N=20, 50$ and the value of the graphon mean field approximation $\bODE_{k/N,s}(t)$ for $t\in[0,3]$. The plots are generated for the states $s=3,6,10$ and $k=14, 35$.}
    \label{fig:plots_distribution_approximation_bikes_single}
\end{figure}

\begin{figure}[htbp]
\includegraphics[width=1\linewidth]{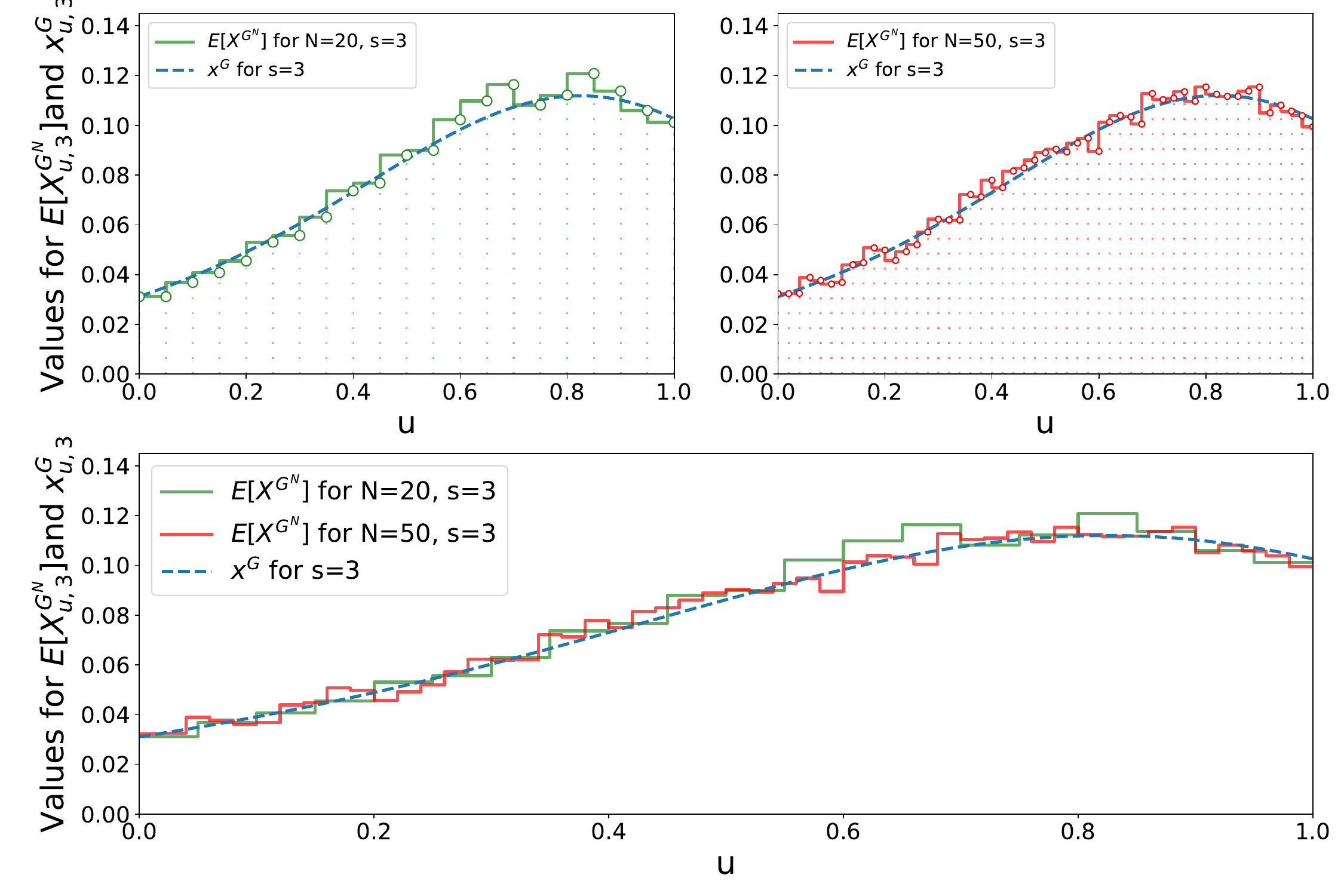}
\caption{For a fixed $t=2.5$ and state $s=3$, the plots compare the graphon mean field approximation $\bODE$ against sample means for systems of size $N=20,50$. As for the main theorem of our results, we represent the values of the stochastic system as step function with constant values on the intervals $(k-1/N,k/N], k\in[N]$. The figure shows that for increasing system size, the values of the stochastic system indeed approach the trajectory of the deterministic system. In the upper plots, the approximation is plotted against values of the sample mean for one system size $N=20, 50$ each. In the lower plot, both functions are overlain for better comparison.}
\label{fig:ode_state_approximation}
\end{figure}




\section{Proofs}
This section provides the proof of the main statements of this paper.



\subsection{Proof of Lemma~\ref{lemma:unique_solution}}
\label{proof:unique_solution}

\begin{proof}{Proof.} We prove the uniqueness and continuous differentiability of the differential equation \eqref{eq:definition_ode} as well as the Lipschitz properties of the drift. The proof is the consequence of two lemmas: 
\begin{itemize}
    \item We show in Lemma \ref{lemma:drift_lipschitz} below that the drift is locally Lipschitz continuous in $L_2$. By~\cite{driverAnalysisToolsApplications2003}, this implies the existence of a local solution and the uniqueness of this solution. 
    \item In Lemma~\ref{lemma:directional_derivative_drift} we show that the directional derivatives of $\FG$ are well-defined and Lipschitz continuous. This property ensures that the $\bODE$ is also continuously differentiable (for a proof of this property in general Banach spaces see for example \cite{driverAnalysisToolsApplications2003}). \hfill \Halmos
\end{itemize}

\end{proof}

\begin{lemma}[Local Lipschitz Continuity of the Drift] \label{lemma:drift_lipschitz}
    Let $L_{\bFG} = 2(C_R \abs{\calS} + 2 C_R \abs{\calS}^2 \opnorm{G})$ where $C_R$ is the bound on the rate functions, \opnorm{G} the operator norm of the graphon $G$ and $\abs{\calS}$ the size of the finite state space. 
    Then for all $\bx_{.}=(x_{.,s})_{s\in\calS}, \by_{.}=(y_{.,s})_{s\in\calS}$ with $x_{.,s}, y_{.,s} \in L_2(0,1]$ and $\norm{x}, \norm{y} \leq 1$ with  $\abs{x_s(u)},\abs{y_{u,s}} \leq 1$ almost everywhere\footnote{Here, \textit{almost everywhere} means that the property holds except on a subset of $(0,1]$ with measure zero.}, we have: 
    \begin{align*}
        \norm{\bFG(\bx) - \bFG(\by)}_{L_2} \leq L_{\bFG} \norm{\bx - \by}_{L^2}.
    \end{align*}
\end{lemma}

\begin{proof}{Proof.}
Define $\hat{\pmb{F}}^G_{u,s}(\bx,\by) := \bR^{\text{uni}}_{u,s}\bx_u + \bx_u^T \int_0^1 \bR^{\text{pair}}_{u,v,s}  G_{u,v}  \by_v \ dv$, i.e., we replace $\bx_v$ by $\by_{v}$ under the integral. Applying the triangle inequality to the $L_2$ norm gives
\begin{align*}
\norm{\bFG(\bx) - \bFG(\by)}_{L_2} & = \norm{\bFG(\bx) - \hat{\pmb{F}}^G(\bx,\by) + \hat{\pmb{F}}^G(\bx,\by) - F(\by)}_{L_2} \\
&  \leq \norm{\bFG(\bx) - \hat{\pmb{F}}^G(\bx,\by)}_{L_2} + \norm{\hat{\pmb{F}}^G(\bx,\by) - \bFG(\by)}_{L_2}.
\end{align*}
Writing out the definitions and using the bounds of the rate vectors / matrices and the bound of the graphon one immediately obtains
\begin{align*}
\norm{\bFG(\bx) - \hat{\pmb{F}}^G(\bx,\by)}_{L_2} & = 
\sqrt{\sum_{s\in\calS} \int_0^1 \left(\bR^{\text{uni}}_{u,s}\bx_u + \bx_u^T \int_0^1 \bR^{\text{pair}}_{u,v,s} G_{u,v}\bx_v dv - \bR^{\text{uni}}_{u,s}\bx_u - \bx_u^T \int_0^1 \bR^{\text{pair}}_{u,v,s}  G_{u,v}  \by_v \ dv \right)^2 du} \\
& = \sqrt{\sum_{s\in\calS} \int_0^1 \left( \bx_u^T \int_0^1 \bR^{\text{pair}}_{u,v,s} G_{u,v}( \bx_v - \by_v) dv \right)^2 du} \leq 2 \abs{\calS}^3 C_R \opnorm{G} \norm{\bx - \by }_{L_2}
\end{align*}
and
\begin{align*}
\norm{\hat{\pmb{F}}^G(\bx,\by) - \bFG(\by)}_{L_2} & = \sqrt{ \sum_{s\in\calS}\int_0^1  \left( \bR^{\text{uni}}_{u,s}(\bx_u - \by_u) + (\bx_u - \by_u)^T \int_0^1 \bR^{\text{pair}}_{u,v,s}  G_{u,v} \by_v \ dv \right)^2 du} \\
& \leq \bigl(2 C_R \abs{\calS}^2 + 2 \abs{\calS}^3 C_R\opnorm{G}\bigr)\norm{\bx - \by}_{L_2}.
\end{align*}
Last, we conclude that $\norm{\FG(\bx) - \FG(\by)}_{L_2} \leq (2(C_R \abs{\calS}^2 + 2 C_R \abs{\calS}^3 \opnorm{G}))\norm{\bx - \by}_{L_2}$, where the $\abs{\calS}^2$ and $\abs{\calS}^3$ terms come from the vector and matrix notation used for the unilateral and pairwise rates as well as the sum over $\calS$. 
{\hfill \Halmos} 
\end{proof}

\begin{lemma}[Lipschitz Continuous Directional Derivative of $\bFG$] \label{lemma:directional_derivative_drift}
The directional derivative of the drift $\FG$ for $\bx=(x_s)_{s\in\calS}$ with $x_s \in L_2(0,1]$ is given by
\begin{align}
D_x\FG_{u,s}(\bx)(\bh) := \bR^{\text{uni}}_{u,s'}\bh_u + \bh_u^T \int_0^1 \bR^{\text{pair}}_{u,v,s}  G_{u,v}\bx_v dv + \bx_u^T \int_0^1 \bR^{\text{pair}}_{u,v,s} G_{u,v}\bh_v dv \label{eq:def_dir_deriv_FG}
\end{align}
where $\bR^{\text{uni}}_{u,s}, \bR^{\text{pair}}_{u,v,s}$ and $G$ are Lebesgue integrable functions. 
Furthermore $D_x\FG_{u,s}(\bx)(h)$ is Lipschitz-continuous in $x$.
\end{lemma}
\begin{proof}{Proof.}
Recall the definition of $\bFG(\bx) = \bR^{\text{uni}}_{u,s'}\bx_{u} + \bx_u^T \int_0^1 \bR^{\text{pair}}_{u,v,s'} G_{u,v}\bx_v \ dv $. 
We define $D_x\FG_{u,s}(\bx)(\bh)$ for $(u,s) \in (0,1] \times \calS$ as in Equation \eqref{eq:def_dir_deriv_FG}.
To see that $D_x\bFG(x)(h)$ is a directional derivative of $\bFG$ in $\bx\in L_2(0,1]^{\calS}$ in direction $h \in L_2$, we show that it fulfills
\begin{align*}
\lim_{\epsilon \downarrow 0} \frac{1}{\epsilon} \norm { \bFG(\bx + \epsilon \bh) - \bFG(\bx) - D_x\bFG(\bx)(\epsilon \bh)}_{L_2} = 0.
\end{align*}
By definition
\begin{align*}
& \norm{ \bFG(\bx + \epsilon \bh) - \bFG(x) - D_x\bFG(\bx)(\bh) }^2_{L_2} \\
& = \int_0^1 \sum_{s} \left( \bR^{\text{uni}}_{u,s}(\bx_{u} + \epsilon \bh_u) + (\bx_u+ \epsilon \bh_u)^T  \int_0^1 \bR^{\text{pair}}_{u,v,s} G_{u,v}(\bx_v + \epsilon \bh_v)\ dv \right. \\
& \qquad - \bR^{\text{uni}}_{u,s'}\bx_{u} - \bx_u^T  \int_0^1 \bR^{\text{pair}}_{u,v,s'} G_{u,v}\bx_v \ dv \\
& \qquad \left. - \epsilon \bR^{\text{uni}}\bh_u - \epsilon \bh_u^T  \int_0^1 \bR^{\text{pair}}_{u,v,s} G_{u,v}\bx_v dv -\epsilon  \bx_u^T  \int_0^1 \bR^{\text{pair}}_{u,v,s} G_{u,v}\bh_v dv \right)^2 du\\
& = \int_0^1 \sum_{s} \left( \epsilon \bR^{\text{uni}}_{u,s} \bh_u + \epsilon \bx_u^T  \int_0^1\bR^{\text{pair}}_{u,v,s} G_{u,v} \bh_v dv + \epsilon \bh_u^T  \int_0^1\bR^{\text{pair}}_{u,v,s} G_{u,v}\bx_v \ dv \right. \\
& \qquad + \epsilon^2 \bh_u^T  \int_0^1\bR^{\text{pair}}_{u,v,s} G_{u,v} \bh_v\ dv \\
& \qquad \left. - \epsilon ( \bR^{\text{uni}}\bh_u - \bh_u^T  \int_0^1 \bR^{\text{pair}}_{u,v,s} G_{u,v}\bx_v dv - \bx_u^T  \int_0^1 \bR^{\text{pair}}_{u,v,s} G_{u,v}\bh_v dv ) \right)^2 du \\
& = \int_0^1 \sum_{s} \left( \epsilon^2 \bh_u^T  \int_0^1 \bR^{\text{pair}}_{u,v,s} G_{u,v} \bh_v dv \right)^2 du.
\end{align*}
Taking dividing by $\epsilon$ and taking $\epsilon \to 0$ shows that $D_x\bFG(\bx)(\bh)$ is indeed a directional derivative of $\bFG$ in $\bx$. By definition, $D_x\bFG_{u,s}(\bx)(\bh)$ is linear in $\bx$ and as $\bR^{\text{uni}}_{u,s}, \bR^{\text{pair}}_{u,v,s}$ and $G_{u,v}$ are bounded functions, the Lipschitz continuity in $L_2$ follows.
{\hfill \Halmos}
\end{proof}

\subsection{Proof of Theorem \ref{thrm:approximation_accuracy}} \label{proof:approximation_accuracy}

In the following, we present the proof of Theorem \ref{thrm:approximation_accuracy}.
\begin{proof}{Proof.}
Let $\bXN_t$ be the $L_2(0,1]^{\abs{\calS}}$ representation of the stochastic system as defined in \ref{sec:representation_X_G_F_L_2}. For a fixed $t\geq0$, define 
\begin{align}
\bnu^N(\tau) = \E[\bODE\bigl(t-\tau, \bXN(\tau)\bigr)] \label{eq:def_nu}
\end{align}
for which we suppress the dependence of the expectation on the graph $G^N$ and initial state $\bXN(0)$ from here on. Using the definition of $\bnu^N$ we rewrite
\begin{align}
\norm{ \E[\bXN(t)] - \ODE(t,\bXN(0)) }_{L_2} & = 
\sqrt{\sum_{s\in\calS} \int_0^1 \left( \E[\XN_{u,s}(t)] - \ODE_{u,s}(t,\bXN(0))\right)^2 du } \notag \\
& = \sqrt{\sum_{s\in\calS} \int_0^1 \left( \nu_{u,s}^N(t) - \nu_{u,s}^N(0) \right)^2du}. \label{eq:ode_stoch_comparison_nu}
\end{align}
To compare the drift of the ODE against the drift of the stochastic process and ultimately bound them, we first want to rewrite $\nu_{u,s}^N(t) - \nu_{u,s}^N(0) = \int_0^t \frac{d}{d\tau} \nu^N_{u,s}(\tau)$, with $\frac{d}{d\tau} \nu^N(\tau)$ being the quantity that fulfills $\bnu^N(b) - \bnu^N(a) = \int_a^b \frac{d}{d\tau} \bnu^N(\theta) d\theta$ for arbitrary $a,b\in[0,t]$ with $a<b$. In Lemma \ref{lemma:nu_derivative} we show that $\frac{d}{d\tau} \bnu
(\theta)$ exists almost everywhere for $\theta \in (0,t)$ and is almost everywhere equal to 
\begin{align}
& \E\left[ \sum_{s'} \int_0^1 \left[ D_{x} \bODE(t-\tau, \bXN(\tau)) \left( \bFG(\bXN(\tau)) - \bFN(\bXN(\tau)) \right) \right]_{v,s'}  dv\right] \label{eq:nu_derivative_drift_difference} \\
& \qquad + \E\left[ \tilde{R}_1(\bODE(t-\tau, \bXN(\tau))) \right]d\tau. \label{eq:nu_derivative_remainder_term}
\end{align} 
In the above sum, $D_{x}\bODE(\tau, \bXN(\tau))\left( \bFG(\bXN(\tau)) - \bFN(\bXN(\tau)) \right)$ is the directional derivative of $\bODE$ in its initial condition in direction $\bFG(\bXN(\tau)) - \bFN(\bXN(\tau))$. 
We aim to bound the sum by using the properties of the derivative of $\bODE$ and the differences between the drift of the deterministic and stochastic system with respect to the $L_2$ norm. The technical details to bound the remainder term and the difference between the drifts are moved to the Lemmas \ref{lemma:remainder_bound} and \ref{lemma:drift_difference_bound}. From the application of the latter lemmas, one obtains the bounds 
\begin{align*}
C_{ D_{x(0)}} \left(  \frac{2 L_{\bR^{\text{pair}}} + 16C^2_{\bR^{\text{pair}}}K_{\bR^{\text{pair}}} }{N} - \frac{16C^2_{\bR^{\text{pair}}}K^2_{\bR^{\text{pair}}} }{N^2} + C_{\bR^\text{pair}}^2 \abs{\calS}^2 \opnorm{G - G^N} \right)
\end{align*}
for Equation \eqref{eq:nu_derivative_drift_difference} and $C_{\tilde{R}}/N$ for Equation \eqref{eq:nu_derivative_remainder_term}. It follows by applying the bounds and rearranging terms, that
\begin{align*}
\norm{ \E[\bXN(t)] - \ODE(t,\bXN(0)) }_{L_2} \leq  \frac{C_A}{N} + C_B \opnorm{G^N -G}
\end{align*} 
for some finite constants $C_A, C_B >0$ additionally to the previous bound also depend on $t$. This concludes the proof.
{\hfill \Halmos}
\end{proof}

\subsection{Proof of Corollary \ref{corollary:graphon_results} and Corollary \ref{corollary:het_approx_results}}
\label{proof:graphon_results}
\label{proof:het_approx_results}

In this section, we prove the corollaries of Theorem~\ref{thrm:approximation_accuracy}. In each case, we use an additional lemma to obtain a bound on $\opnorm{G^N-G}$. 

\begin{proof}{Proof.}[Proof of Corollary \ref{corollary:graphon_results}]
By the application of Theorem \ref{thrm:approximation_accuracy} we have
\begin{align*}
\norm{ \E[\bXN(t)] - \ODE(t,\bXN(0)) }_{L_2} & =  \frac{C_A}{N} + C_B \opnorm{G^N -G}.
\end{align*}
It therefore remains to bound the $L_2$ distance  $\opnorm{G^N -G}$ between the graphon $G$ and graph $G^N$. By application of Lemma~\ref{lemma:graphon_distance} with probability at least $1 - \delta$ and for `large enough' $N$ 
\begin{align*}
\opnorm{G^N -G} \leq \sqrt{\frac{4 \log{(2N/\delta)}}{N}} + 2\sqrt{\frac{(L_G^2 - K_G^2)}{N^2} +\frac{K_G}{N}} =: \psi_{\delta, G}(N).
\end{align*}
Defining and substitution $\delta = 2/N$ into the above equation concludes the proof. {\hfill \Halmos}

\end{proof}


\begin{proof}{Proof.}[Proof of Corollary \ref{corollary:het_approx_results}]
The corollary is a direct implication of Theorem \ref{corollary:graphon_results} and Lemma \ref{lemma:lipschtiz_discretization_error}: due to $G^N$ representing a discretized version of $G$ it directly follows that $\opnorm{G - G^N}$ is of order $\mathcal{O}(1/N)$, which concludes the proof. 
{\hfill \Halmos}
\end{proof}

\section{Conclusion}

In this paper, we study an approximation for a system of particles interacting on a graph. We show that when the interacting graph converges to a graphon, the underlying behavior of the stochastic system converges to a deterministic limit, which we call the graphon mean field approximation. While this result is similar to other results in the literature -- that show that graphon mean field approximations are asymptotically exact in some settings --, our main contribution is to provide precise bounds on the accuracy of this approximation. We showed indeed that the distance between the original finite-$N$ system and the graphon mean field approximation can be bounded by a term $\mathcal{O}(1/N)$ that depends on the number of particles of the system plus a term $\mathcal{O}(\opnorm{G^N-G})$ that depends on the distance between the original graph $G^N$ and the graphon $G$, when measured as a $L_2$ operator.

This paper aims to be methodological. It shows that it is possible to obtain bounds for a system with a graph structure, and not mere asymptotic convergence results. To keep the presentation reasonable, we intentionally considered a relatively simple model, for instance by restricting our attention to pairwise interactions between nodes or unilateral jumps, and by considering that the rate functions $r^N$ are discretized versions of the limiting rates $r$. We believe that by doing so, the proof is easier to follow and could then be adapted to more general cases. 

The focus of this paper is to study the case of dense graphs that converge to graphons. This implies that the total number of edges per node is of order $\mathcal{O}(N)$. We believe that our methodology could be applicable in the not-so-dense case when the number of edges per node goes to infinity at a sub-linear rate. This would probably give bounds that converge more slowly to zero.  Complementary to this paper, another interesting question is the case of sparse graphs, where the number of neighbors per node remains bounded when the number of nodes $N$ goes to infinity. Yet, this would require a fundamentally different approach: studying such a problem is out of the scope of our tools since the graphon mean field approximation is not asymptotically exact for sparse graphs. 


\bibliographystyle{informs2014}
\bibliography{refs}

\begin{thebibliography}{38}
\providecommand{\natexlab}[1]{#1}
\providecommand{\url}[1]{\texttt{#1}}
\providecommand{\urlprefix}{URL }

\bibitem[{Abbe(2018)}]{abbeCommunityDetectionStochastic2018}
Abbe E (2018) Community {{Detection}} and {{Stochastic Block Models}}. \emph{Foundations and Trends{\textregistered} in Communications and Information Theory} 14(1-2):1--162, ISSN 1567-2190, 1567-2328, \urlprefix\url{http://dx.doi.org/10.1561/0100000067}.

\bibitem[{Allmeier \protect\BIBand{} Gast(2022)}]{allmeierMeanFieldRefined2022}
Allmeier S, Gast N (2022) Mean {{Field}} and {{Refined Mean Field Approximations}} for {{Heterogeneous Systems}}: {{It Works}}! \emph{Proceedings of the ACM on Measurement and Analysis of Computing Systems} 6(1):13:1--13:43, \urlprefix\url{http://dx.doi.org/10.1145/3508033}.

\bibitem[{Aurell et~al.(2022)Aurell, Carmona, Dayan{\i}kl{\i}, \protect\BIBand{} Lauri{\`e}re}]{aurellFiniteStateGraphon2022}
Aurell A, Carmona R, Dayan{\i}kl{\i} G, Lauri{\`e}re M (2022) Finite {{State Graphon Games}} with {{Applications}} to {{Epidemics}}. \emph{Dynamic Games and Applications} 12(1):49--81, ISSN 2153-0793, \urlprefix\url{http://dx.doi.org/10.1007/s13235-021-00410-2}.

\bibitem[{{Avella-Medina} et~al.(2018){Avella-Medina}, Parise, Schaub, \protect\BIBand{} Segarra}]{avella-medinaCentralityMeasuresGraphons2018}
{Avella-Medina} M, Parise F, Schaub MT, Segarra S (2018) Centrality measures for graphons: {{Accounting}} for uncertainty in networks. \urlprefix\url{http://dx.doi.org/10.1109/TNSE.2018.2884235}.

\bibitem[{Bayraktar et~al.(2023)Bayraktar, Chakraborty, \protect\BIBand{} Wu}]{bayraktarGraphonMeanField2023}
Bayraktar E, Chakraborty S, Wu R (2023) Graphon mean field systems. \emph{The Annals of Applied Probability} 33(5):3587--3619, ISSN 1050-5164, 2168-8737, \urlprefix\url{http://dx.doi.org/10.1214/22-AAP1901}.

\bibitem[{Bayraktar \protect\BIBand{} Wu(2021)}]{bayraktarMeanFieldInteraction2021}
Bayraktar E, Wu R (2021) Mean field interaction on random graphs with dynamically changing multi-color edges. \emph{Stochastic Processes and their Applications} 141:197--244, ISSN 0304-4149, \urlprefix\url{http://dx.doi.org/10.1016/j.spa.2021.07.005}.

\bibitem[{Bet et~al.(2024)Bet, Coppini, \protect\BIBand{} Nardi}]{betWeaklyInteractingOscillators2024}
Bet G, Coppini F, Nardi FR (2024) Weakly interacting oscillators on dense random graphs. \emph{Journal of Applied Probability} 61(1):255--278, ISSN 0021-9002, 1475-6072, \urlprefix\url{http://dx.doi.org/10.1017/jpr.2023.34}.

\bibitem[{Bhamidi et~al.(2019)Bhamidi, Budhiraja, \protect\BIBand{} Wu}]{bhamidiWeaklyInteractingParticle2019}
Bhamidi S, Budhiraja A, Wu R (2019) Weakly interacting particle systems on inhomogeneous random graphs. \emph{Stochastic Processes and their Applications} 129(6):2174--2206, ISSN 0304-4149, \urlprefix\url{http://dx.doi.org/10.1016/j.spa.2018.06.014}.

\bibitem[{Blanchard \protect\BIBand{} Br{\"u}ning(2015)}]{blanchardMathematicalMethodsPhysics2015}
Blanchard P, Br{\"u}ning E (2015) \emph{Mathematical {{Methods}} in {{Physics}}: {{Distributions}}, {{Hilbert Space Operators}}, {{Variational Methods}}, and {{Applications}} in {{Quantum Physics}}}, volume~69 of \emph{Progress in {{Mathematical Physics}}} (Cham: Springer International Publishing), ISBN 978-3-319-14044-5, \urlprefix\url{http://dx.doi.org/10.1007/978-3-319-14045-2}.

\bibitem[{Braverman \protect\BIBand{} Dai(2017)}]{bravermanSteinMethodSteadystate2017b}
Braverman A, Dai JG (2017) Stein's method for steady-state diffusion approximations of {{M}}/{{Ph}}/n+{{M}} systems. \emph{The Annals of Applied Probability} 27(1), ISSN 1050-5164, \urlprefix\url{http://dx.doi.org/10.1214/16-AAP1211}.

\bibitem[{Braverman et~al.(2017)Braverman, Dai, \protect\BIBand{} Feng}]{bravermanSteinMethodSteadystate2017c}
Braverman A, Dai {\relax JG}, Feng J (2017) Stein's method for steady-state diffusion approximations: An introduction through the {{Erlang-A}} and {{Erlang-C}} models. \emph{Stochastic Systems} 6(2):301--366.

\bibitem[{Budhiraja et~al.(2019)Budhiraja, Mukherjee, \protect\BIBand{} Wu}]{budhirajaSupermarketModelGraphs2019}
Budhiraja A, Mukherjee D, Wu R (2019) Supermarket model on graphs. \emph{The Annals of Applied Probability} 29(3), ISSN 1050-5164, \urlprefix\url{http://dx.doi.org/10.1214/18-AAP1437}.

\bibitem[{Caines \protect\BIBand{} Huang(2021)}]{cainesGraphonMeanField2021}
Caines PE, Huang M (2021) Graphon {{Mean Field Games}} and {{Their Equations}}. \emph{SIAM Journal on Control and Optimization} 59(6):4373--4399, ISSN 0363-0129, 1095-7138, \urlprefix\url{http://dx.doi.org/10.1137/20M136373X}.

\bibitem[{Decreusefond et~al.(2012)Decreusefond, Dhersin, Moyal, \protect\BIBand{} Tran}]{decreusefondLargeGraphLimit2012}
Decreusefond L, Dhersin JS, Moyal P, Tran VC (2012) Large graph limit for an {{SIR}} process in random network with heterogeneous connectivity. \emph{The Annals of Applied Probability} 22(2):541--575, ISSN 1050-5164, 2168-8737, \urlprefix\url{http://dx.doi.org/10.1214/11-AAP773}.

\bibitem[{Delmas et~al.(2023)Delmas, Frasca, Garin, Tran, Velleret, \protect\BIBand{} Zitt}]{delmasIndividualBasedSIS2023}
Delmas JF, Frasca P, Garin F, Tran VC, Velleret A, Zitt PA (2023) Individual based {{SIS}} models on (not so) dense large random networks.

\bibitem[{Driver(2003)}]{driverAnalysisToolsApplications2003}
Driver BK (2003) Analysis tools with applications. \emph{Lecture notes} .

\bibitem[{Fricker \protect\BIBand{} Gast(2016)}]{frickerIncentivesRedistributionHomogeneous2016}
Fricker C, Gast N (2016) Incentives and redistribution in homogeneous bike-sharing systems with stations of finite capacity. \emph{EURO Journal on Transportation and Logistics} 5(3):261--291, ISSN 2192-4384, \urlprefix\url{http://dx.doi.org/10.1007/s13676-014-0053-5}.

\bibitem[{Fricker et~al.(2012)Fricker, Gast, \protect\BIBand{} Mohamed}]{frickerMeanFieldAnalysis2012}
Fricker C, Gast N, Mohamed H (2012) Mean field analysis for inhomogeneous bike sharing systems. \emph{{{AofA}}}, volume DMTCS Proceedings vol. AQ, 23rd Intern. Meeting on Probabilistic, Combinatorial, and Asymptotic Methods for the Analysis of Algorithms (AofA'12) (Montreal, Canada: DMTCS), \urlprefix\url{http://dx.doi.org/10.46298/dmtcs.3006}.

\bibitem[{Ganguly(2022)}]{gangulyNonMarkovianInteractingParticle2022}
Ganguly A (2022) \emph{Non-{{Markovian Interacting Particle Systems}} on {{Large Sparse Graphs}}: {{Hydrodynamic Limits}} and {{Marginal Characterizations}}}. Ph.D. thesis, Brown Univerity.

\bibitem[{Ganguly \protect\BIBand{} Ramanan(2022)}]{gangulyHydrodynamicLimitsNonMarkovian2022}
Ganguly A, Ramanan K (2022) Hydrodynamic {{Limits}} of non-{{Markovian Interacting Particle Systems}} on {{Sparse Graphs}}.

\bibitem[{Gast(2017)}]{gastExpectedValuesEstimated2017}
Gast N (2017) Expected {{Values Estimated}} via {{Mean-Field Approximation}} are 1/{{N-Accurate}}. \emph{Proceedings of the ACM on Measurement and Analysis of Computing Systems} 1(1):1--26, ISSN 24761249, \urlprefix\url{http://dx.doi.org/10.1145/3084454}.

\bibitem[{Gast \protect\BIBand{} Van~Houdt(2017)}]{gastRefinedMeanField2017}
Gast N, Van~Houdt B (2017) A {{Refined Mean Field Approximation}}. \emph{Proceedings of the ACM on Measurement and Analysis of Computing Systems} 1(2):33:1--33:28, \urlprefix\url{http://dx.doi.org/10.1145/3154491}.

\bibitem[{Keliger(2023)}]{keligerUniversalitySisEpidemics2023}
Keliger D (2023) Universality of {{Sis Epidemics Starting}} from {{Small Initial Conditions}}. \urlprefix\url{http://dx.doi.org/10.2139/ssrn.4543174}.

\bibitem[{Kurtz(1970)}]{kurtzSolutionsOrdinaryDifferential1970}
Kurtz TG (1970) Solutions of ordinary differential equations as limits of pure jump markov processes. \emph{Journal of Applied Probability} 7(1):49--58, ISSN 0021-9002, 1475-6072, \urlprefix\url{http://dx.doi.org/10.2307/3212147}.

\bibitem[{Kurtz(1971)}]{kurtzLimitTheoremsSequences1971}
Kurtz TG (1971) Limit theorems for sequences of jump {{Markov}} processes approximating ordinary differential processes. \emph{Journal of Applied Probability} 8(2):344--356, ISSN 0021-9002, 1475-6072, \urlprefix\url{http://dx.doi.org/10.2307/3211904}.

\bibitem[{Le~Boudec et~al.(2007)Le~Boudec, McDonald, \protect\BIBand{} Mundinger}]{leboudecGenericMeanField2007}
Le~Boudec JY, McDonald D, Mundinger J (2007) A {{Generic Mean Field Convergence Result}} for {{Systems}} of {{Interacting Objects}}. \emph{Fourth {{International Conference}} on the {{Quantitative Evaluation}} of {{Systems}} ({{QEST}} 2007)}, 3--18 (Edinburgh, Scotland, UK: IEEE), ISBN 978-0-7695-2883-0, \urlprefix\url{http://dx.doi.org/10.1109/QEST.2007.8}.

\bibitem[{Lov{\'a}sz(2012)}]{lovaszLargeNetworksGraph2012}
Lov{\'a}sz L (2012) \emph{Large {{Networks}} and {{Graph Limits}}}, volume~60 of \emph{Colloquium {{Publications}}} (Providence, Rhode Island: American Mathematical Society), ISBN 978-0-8218-9085-1, \urlprefix\url{http://dx.doi.org/10.1090/coll/060}.

\bibitem[{McKean(1967)}]{mckeanPropagationChaosClass1967}
McKean HP (1967) Propagation of chaos for a class of non-linear parabolic equations. \emph{Stochastic Differential Equations (Lecture Series in Differential Equations, Session 7, Catholic Univ., 1967)} 41--57.

\bibitem[{Mitzenmacher(2001)}]{mitzenmacherPowerTwoChoices2001}
Mitzenmacher M (2001) The power of two choices in randomized load balancing. \emph{IEEE Transactions on Parallel and Distributed Systems} 12(10):1094--1104, ISSN 10459219, \urlprefix\url{http://dx.doi.org/10.1109/71.963420}.

\bibitem[{Norman(1972)}]{normanMarkovProcessesLearning1972}
Norman MF (1972) \emph{Markov Processes and Learning Models}, volume~84 (Academic Press New York), ISBN 0-12-521450-2.

\bibitem[{Ramanan(2022)}]{ramananMeanfieldLimitsAnalysis2022}
Ramanan K (2022) Beyond mean-field limits for the analysis of large-scale networks. \emph{Queueing Systems} 100(3-4):345--347, ISSN 0257-0130, 1572-9443, \urlprefix\url{http://dx.doi.org/10.1007/s11134-022-09845-9}.

\bibitem[{Roy et~al.(2023)Roy, Singh, \protect\BIBand{} Narahari}]{royRecentAdvancesModeling2023}
Roy A, Singh C, Narahari Y (2023) Recent advances in modeling and control of epidemics using a mean field approach. \emph{S{\=a}dhan{\=a}} 48(4):207, ISSN 0973-7677, \urlprefix\url{http://dx.doi.org/10.1007/s12046-023-02268-z}.

\bibitem[{Rutten \protect\BIBand{} Mukherjee(2023)}]{ruttenMeanfieldAnalysisLoad2023}
Rutten D, Mukherjee D (2023) Mean-field {{Analysis}} for {{Load Balancing}} on {{Spatial Graphs}}. \emph{Abstract {{Proceedings}} of the 2023 {{ACM SIGMETRICS International Conference}} on {{Measurement}} and {{Modeling}} of {{Computer Systems}}}, 27--28, {{SIGMETRICS}} '23 (New York, NY, USA: Association for Computing Machinery), ISBN 9798400700743, \urlprefix\url{http://dx.doi.org/10.1145/3578338.3593552}.

\bibitem[{Stein(1986)}]{steinApproximateComputationExpectations1986}
Stein C (1986) Approximate computation of expectations. \emph{Lecture Notes-Monograph Series} 7:i--164.

\bibitem[{Van Der~Hofstad(2017)}]{vanderhofstadRandomGraphsComplex2017}
Van Der~Hofstad R (2017) \emph{Random {{Graphs}} and {{Complex Networks}}} (Cambridge: Cambridge University Press), ISBN 978-1-316-77942-2, \urlprefix\url{http://dx.doi.org/10.1017/9781316779422}.

\bibitem[{Vizuete et~al.(2020)Vizuete, Frasca, \protect\BIBand{} Garin}]{vizueteGraphonBasedSensitivityAnalysis2020}
Vizuete R, Frasca P, Garin F (2020) Graphon-{{Based Sensitivity Analysis}} of {{SIS Epidemics}}. \emph{IEEE Control Systems Letters} 4(3):542--547, ISSN 2475-1456, \urlprefix\url{http://dx.doi.org/10.1109/LCSYS.2020.2971021}.

\bibitem[{Zhao \protect\BIBand{} Mukherjee(2024)}]{zhaoOptimalRateMatrixPruning2024}
Zhao Z, Mukherjee D (2024) Optimal {{Rate-Matrix Pruning For Heterogeneous Systems}}. \emph{ACM SIGMETRICS Performance Evaluation Review} 51(4):26--27, ISSN 0163-5999, \urlprefix\url{http://dx.doi.org/10.1145/3649477.3649492}.

\bibitem[{Zhao et~al.(2024)Zhao, Mukherjee, \protect\BIBand{} Wu}]{zhaoExploitingDataLocality2024}
Zhao Z, Mukherjee D, Wu R (2024) Exploiting {{Data Locality}} to {{Improve Performance}} of {{Heterogeneous Server Clusters}}. \emph{Stochastic Systems} ISSN 1946-5238, \urlprefix\url{http://dx.doi.org/10.1287/stsy.2022.0040}.

\end{thebibliography}

\begin{appendices}

\section{Lemmas}

This section contains the most technical lemmas of the paper. 

\subsection{Derivative of \texorpdfstring{$\nu$}{ν} (Lemma~\ref{lemma:nu_derivative})}

The following Lemma derives a representation  of the derivative of $\nu$, as defined in the proof of Theorem \ref{thrm:approximation_accuracy}, which depends on the derivative of the graphon mean field approximation with respect to the initial condition, the transitions of the stochastic system and the deterministic drift. 

\begin{lemma}\label{lemma:nu_derivative}
For $\nu^N(\tau)$ and $\tau \in (0,t)$ we define 
\begin{align*}
\frac{d}{d\tau} \nu^N(\tau) := \lim_{h\downarrow 0} \frac{1}{h}\int_\tau^{\tau+h} \lim_{r\downarrow0} \frac{1}{r} \left( \E[\bODE(t-(\theta +r), \bXN(\theta +r))] - \E[\bODE(t-\theta, \bXN(\theta))]  \right)d\theta.
\end{align*}
for which, by construction and continuity properties of $\bODE$ and $\bXN$, it holds that for arbitrary $a, b \in [0,T], a<b$, $\int_a^b \frac{d}{d\tau} \nu^N(\tau) = \nu^N(b) - \nu^N(a)$. Furthermore, the right-hand side is almost everywhere equal to
\begin{align*}
& \E[ \sum_{s'} \int_0^1 \left[ D_{x} \bODE(t-\tau, \bXN(\tau))  \left( \bFG(\bXN(\tau)) - \bFN(\bXN(\tau)) \right) \right]_{v,s} du ] \\
& \quad + \E [ \tilde{R}_1(\bODE(t-\tau, \bXN(\tau))) ].
\end{align*}
which is well defined and bounded due to the properties of the drifts $\bFN, \bFG$ and the approximation $\bODE$. 
In this context, $\tilde{R}_1$ refers to the summed Taylor remainder term defined by 
\begin{align*}
\tilde{R}_1\bigl(\bODE(t-\tau, \bXN(\tau))\bigr) := & \sum_{k,s_k,s_k'} R_1(\bODE(t-\tau, \bXN(\tau)), \BFe_{k,s_k'}^N - \BFe_{k,s_k}^N) \\
& \qquad \times \left[ \XN_{k,s_k}r^{N,\text{uni}}_{k,s_k\to s_k'} + X_{k,s_k} \sum_{l\in[N]} \sum_{s_l} r^{N,\text{pair}}_{k,l,s_k \to s_k',s_l} \frac{G^N_{kl}}{N} X_{l,s_l} \right].
\end{align*}
\end{lemma}
\begin{proof}{Proof.}
Recall, that $\bXN$ is in the following handled as a $L_2$ function as described in Section \ref{sec:representation_X_G_F_L_2}. We start the proof by rewriting
\begin{align*}
\frac{d}{d\tau} \nu(\tau) & := \lim_{h\downarrow 0} \frac{1}{h}\int_\tau^{\tau+h} \lim_{r\downarrow0} \frac{1}{r} \left( \E[\bODE(t-(\theta +r), \bXN(\theta +r))] - \E[\bODE(t-\theta, \bXN(\theta))]  \right)d\theta \\
&= \lim_{h\downarrow 0} \frac{1}{h}\int_\tau^{\tau+h} \lim_{r\downarrow0} \frac{1}{r} \left( \E[\bODE(t-(\theta +r), \bXN(\theta +r))] - \E[\bODE(t-\theta, \bXN(\theta +r))]  \right) d\theta\\
& \quad + \lim_{h\downarrow 0} \frac{1}{h}\int_\tau^{\tau+h} \lim_{r\downarrow0}\frac{1}{r} \left( \E[\bODE(t-\theta, \bXN(\theta +r))] - \E[\bODE(t-\theta, \bXN(\theta))]  \right) d\theta \\
&= \lim_{h\downarrow 0} \frac{1}{h}\int_\tau^{\tau+h} \lim_{r\downarrow0} \frac{1}{r} \left( \E\left[ \bODE(t-\theta, \bODE(-r,\bXN(\theta +r)) ) - \bODE(t-\theta, \bXN(\theta +r))] \right]  \right) d\theta \\
& \quad + \lim_{h\downarrow 0} \frac{1}{h}\int_\tau^{\tau+h} \lim_{r\downarrow0}\frac{1}{r} \left(  \E\left[ \E[\bODE(t-\theta, \bXN(\theta +r)) - \bODE(t-\theta, \bXN(\theta))\mid \bXN(\tau)] \right] \right) d\theta \\ 
&= \lim_{h\downarrow 0} \frac{1}{h}\int_\tau^{\tau+h} \E\left[\lim_{r\downarrow0} \frac{1}{r} \left( \bODE(t-\theta, \bODE(-r,\bXN(\theta +r)) ) - \bODE(t-\theta, \bXN(\theta +r))  \right) \right] d\theta \\
& \quad + \lim_{h\downarrow 0} \frac{1}{h}\int_\tau^{\tau+h}  \E\left[ \lim_{r\downarrow0} \frac{1}{r} \left(  \E[ \bODE(t-\theta, \bXN(\theta +r)) - \bODE(t-\theta, \bXN(\theta)) \mid \bXN(\tau)] \right) \right] d\theta
\end{align*}
For the first equality, an artificial zero is added by adding and subtracting the term $\bODE(t-\theta, \bXN(\theta +r))$ and using dominated convergence in combination with the boundedness of the solution of the differential equation and the stochastic system almost everywhere. For the second equality, we use the tower property to write the difference in the conditional expectation with respect to $\bXN(\tau)$. The last equality, follows by dominated convergence which allows to take the limit inside the outer expectation and by definition of $\bODE$ as well as the drift $\bFG$ which allows to rewrite $ \bODE(t-(\theta + r), \bXN(\theta +r)) = \bODE(t-\theta, \bODE(-r,\bXN(\theta +r)))$. By definition, $\bODE$ is continuously differentiable in $L_2$ with respect to its initial condition. The property follows by a general version of Grönwalls inequality and the Lipschitz properties of the drift, also see 
\cite{driverAnalysisToolsApplications2003} for a proof of this property. Using the differentiability in combination with the Lebesgue differentiation theorem for $L_p$ spaces and application of the chain rule we can rewrite
\begin{align}
& \lim_{h\downarrow 0} \frac{1}{h}\int_\tau^{\tau+h} \E\left[ \lim_{r\downarrow0} \frac{1}{r} \left( \bODE(t-\theta, \bODE(-r,\bXN(\theta +r)) ) - \bODE(t-\theta, \bXN(\theta +r))  \right) \right] d\theta \notag \\
&   = \E\left[ D_x \bODE(t-\tau, \bXN(\tau))\bFG(\bXN(\tau)) \right] \label{eq:ode_deriv_initial_cond}
\end{align}
where equality holds almost everywhere and $D_x \bODE(t-\tau, \bXN(\tau))\bFG(\bXN(\tau))$ is the directional derivative of $\bODE$ in its initial condition with direction $\bFG(\bXN(\tau))$. For the second integral, we apply Lebesgue differentiation theorem in combination with the definition of the transitions for the stochastic system. In particular, we use that the probability for a small $d\tau>0$ the state of the stochastic system changes to $\bXN(\tau + d\tau) = \bXN(\tau) + \BFe_{k,s_k'}^N - \BFe_{k,s_k}^N$  is given by $d\tau \XN_{k,s_k}(\tau)r^{N,\text{uni}}_{k,s_k\to s_k'} + d\tau \XN_{k,s_k}(\tau) \sum_{l\in[N]} \sum_{s_l} r^{N,\text{pair}}_{k,l,s_k \to s_k',s_l} \frac{G^N_{kl}}{N} \XN_{l,s_l}(\tau) +  d\tau)$. 
This shows
\begin{align*}
& \lim_{h\downarrow 0} \int_\tau^{\tau+h} \E\left[ \lim_{r\downarrow0}\frac{1}{r} \left( \bODE(t-\theta, \bXN(\theta +r)) - \bODE(t-\theta, \bXN(\theta))  \right) \right] d\theta \\
& = \E\left[ \lim_{d\tau \downarrow 0} \frac{1}{d\tau}  \sum_{k,s_k,s_k'} \left( \bODE(t-\tau, \bXN(\tau) + \BFe_{k,s_k'}^N - \BFe_{k,s_k}^N) - \bODE(t-\tau, \bXN(\tau)) \right) \right. \\
& \qquad \times \left.\left( d\tau \XN_{k,s_k}(\tau)r^{N,\text{uni}}_{k,s_k\to s_k'} + d\tau \XN_{k,s_k}(\tau) \sum_{l\in[N]} \sum_{s_l} r^{N,\text{pair}}_{k,l,s_k \to s_k',s_l} \frac{G^N_{kl}}{N} \XN_{l,s_l}(\tau) + o(d\tau)\right) \right] 
\end{align*}
where equality holds almost everywhere. 
Lastly, by continuous differentiability of $\bODE$ with respect to its initial condition, we apply the $L_2$ version of the Taylor expansion, see Lemma \ref{thrm:taylor_with_remainder_for_banach_spaces}, for $\bODE(t-\tau, \bXN(\tau) + \BFe_{k,s_k'}^N - \BFe_{k,s_k}^N)$ around $\bODE(t-\tau, \bXN(\tau))$. Note again, that the equality for the expansion hold almost everywhere.
This allows to rewrite
\begin{align*}
\bODE(t-\tau, \bXN(\tau) + \BFe_{k,s_k'}^N - \BFe_{k,s_k}^N) = & \bODE(t-\tau, \bXN(\tau))  + \int_0^1 \sum_{s'} \left[ D_{x}\bODE(t-\tau, \bXN(\tau)) \left( \BFe_{k,s_k'}^N - \BFe_{k,s_k}^N\right)\right]_{v,s'} \\ 
& + R_1(\bODE(t-\tau, \bXN(\tau)), \BFe_{k,s_k'}^N - \BFe_{k,s_k}^N).
\end{align*}
Here, $D_{x}\bODE(t-\tau, \bXN(\tau)) \left( \BFe_{k,s_k'}^N - \BFe_{k,s_k}^N\right)$ is the directional derivative of $\bODE$ in direction $\BFe_{k,s_k'}^N - \BFe_{k,s_k}^N$ which should be interpreted as $L_2$ similar as for the $L_2$ representation of $\bXN$.
By $R_1(\bODE(t-\tau, \bXN(\tau)), \BFe_{k,s_k'}^N - \BFe_{k,s_k}^N)$ we denote the Taylor remainder term of first order defined as
\begin{align*}
& R_1(\bODE(t-\tau, \bXN(\tau)), \BFe_{k,s_k'}^N - \BFe_{k,s_k}^N) \\ 
& = \int_0^1(1-\delta)  \int_0^1 \sum_{s'}\left[ \left( D_x\bODE(t-\tau, \bXN(\tau) + \delta (e_{k,s_k'}^N - \BFe_{k,s_k}^N)) - D_x\bODE(t-\tau, \bXN(\tau)) \right) \right. \\
& \qquad  \left. \times (e_{k,s_k'}^N - \BFe_{k,s_k}^N)\right]_{v,s'} \ dv \ d\delta.
\end{align*}
Adding this into the previous right-hand side and rearranging the terms yields by using the linearity of the directional derivative
\begin{align*}
& \E \left[ \sum_{k,s_k,s_k'}\left( \int_0^1 \sum_{s'} \left[ D_{x}\bODE(t-\tau, \bXN(\tau)) ( \BFe_{k,s_k'}^N - \BFe_{k,s_k}^N)\right]_{v,s'}dv + R_1(\bODE(t-\tau, \bXN(\tau)), \BFe_{k,s_k'}^N - \BFe_{k,s_k}^N)\right) \right.\\
& \qquad \times \left.\left( \XN_{k,s_k}(\tau)r^{N,\text{uni}}_{k,s_k\to s_k'} + \XN_{k,s_k}(\tau) \sum_{l\in[N]} \sum_{s_l} r^{N,\text{pair}}_{k,l,s_k \to s_k',s_l} \frac{G^N_{kl}}{N} \XN_{l,s_l}(\tau) \right) \right] \\
& = \E \left[ \int_0^1 \sum_{s'} \left[ D_{x}\bODE(t-\tau, \bXN(\tau)) \bFN(\bXN(\tau)) \right]_{v,s'} dv \right] + \E\left[ \tilde{R}_1\bigl(\bODE(t-\tau, \bXN(\tau))\bigr) \right].
\end{align*}
with 
\begin{align}
\tilde{R}_1\bigl(\bODE(t-\tau, \bXN(\tau))\bigr) := & \sum_{k,s_k,s_k'} R_1(\bODE(t-\tau, \bXN(\tau)), \BFe_{k,s_k'}^N - \BFe_{k,s_k}^N) \label{eq:summed_remainder_def} \\
& \times \left( \XN_{k,s_k}r^{N,\text{uni}}_{k,s_k\to s_k'} + \XN_{k,s_k} \sum_{l\in[N]} \sum_{s_l} r^{N,\text{pair}}_{k,l,s_k \to s_k',s_l} \frac{G^N_{kl}}{N} \XN_{l,s_l} \right) . \notag
\end{align}
Ultimately, we rewrite the initial integral by using the formulations of Equations \eqref{eq:ode_deriv_initial_cond} and \eqref{eq:summed_remainder_def} to
\begin{align*}
& \lim_{h\downarrow 0} \frac{1}{h}\int_\tau^{\tau+h} \lim_{r\downarrow0} \frac{1}{r} \left( \E[\bODE(t-(\theta +r), \bXN(\theta +r))] - \E[\bODE(t-\theta, \bXN(\theta))]  \right)d\theta \\
& = \E[ \sum_{s'} \int_0^1 \left[ D_{x} \bODE(t-\tau, \bXN(\tau)) \left( \FG(\bXN(\tau)) - \bFN(\bXN(\tau)) \right) \right]_{v,s'} dv \mid G^N , \bXN(0)] \\
& \qquad + \E [ \tilde{R}_1(\bODE(t-\tau, \bXN(\tau))) \mid G^N , \bXN(0) ]
\end{align*} 
where equality holds almost everywhere. This concludes the proof. 
{\hfill \Halmos}
\end{proof}

\subsection{Bound on the Taylor Remainder Term (Lemma~\ref{lemma:remainder_bound})}

The subsequent Lemma states an upper bound for the remainder term arising in Lemma \ref{lemma:nu_derivative} which ultimately is used to obtain the bound for Theorem \ref{thrm:approximation_accuracy}. 

\begin{lemma} \label{lemma:remainder_bound}
For any attainable state $\bX\in\calX^N$ of the stochastic process and $\tau\geq0$, the expectation of the Taylor remainder term $\tilde{R}_1 \bigl(\ODE_{u,s}(t,\bX)\bigr)$\footnote{In order for $\bODE_t$ to be properly defined, $\bX$ should be understood in this context as a $L_2(0,1]^{\abs{\calS}}$ function.} as defined in Equation \eqref{eq:summed_remainder_def} is of order $\mathcal{O}(1/N)$.
\end{lemma}

\begin{proof}{Proof.}
For a state $\bX\in\calX^N$, recall the definition of Equation \eqref{eq:summed_remainder_def} 
\begin{align*}
\tilde{R}_1 \bigl(\ODE_{u,s}(\tau,\bX) \bigr) = \sum_{k\in[N], s_k, s_k'\in\calS} & R_1(\ODE_{u,s}(\tau,.),\bX, e^{(k,s_k')} - e^{(k,s_k)})  \\
& \quad \times \bigl( \XN_{k/N,s_k}r^{N,\text{uni}}_{k,s_k\to s_k'} + \XN_{k/N,s_k}  \sum_{s_{l}\in\calS}\int_0^1 r^{\text{pair}}_{k/N,v,s_k\to s_k', s_l}  G^N_{k/N,v} \XN_{v,s_l}dv \bigr)
\end{align*}
where we use that $r^{N,\text{pair}}_{k,l,s_k \to s_k',s_l} = r^{\text{pair}}_{k/N,v,s_k\to s_k', s_l}$ for $v\in(l-1/N, l]$ and $R_1$ given by
\begin{align*}
& R_1(\ODE_{u,s}(\tau,.),\bX, \BFe^N_{(k,s_k)} - \BFe^N_{(k,s_k)}) = \\ 
& \int_0^1 (1-\delta) \sum_{s'} \int_0^1  \left[ \left( D_x \bODE\bigl(\tau-\delta,\bX + \delta(\BFe^N_{(k,s_k)} - \BFe^N_{(k,s_k)})\bigr) - D_{x} \bODE\bigl(\tau,\bX \bigr) \right) (\BFe^N_{(k,s_k')} - \BFe^N_{(k,s_k)})\right]_{v,s'} dv \ d\delta.
\end{align*}
To be clear, as in the proof of Lemma \ref{lemma:nu_derivative}, $\left[ D_{x}\bODE(t, \bx)(\by) \right]_{v,s'}$ refers to the directional derivative of $\bODE$ in its initial condition in direction $\by$ evaluated at $(v,s')$. By the properties of $\bODE$, as stated in Lemma \ref{lemma:unique_solution}, $D_{x} \ODE$ is Lipschitz continuous in with respect to the initial condition. Thus, we can bound
\begin{align*}
& \sum_{s'} \vert\vert \left( D_{x} \ODE \bigl(\tau,\bX + \delta(\BFe^N_{(k,s_k)} -  \BFe^N_{(k,s_k)})) -  D_{x} \ODE_{u,s}\bigl(\tau,\bX\bigr) \right) (\BFe^N_{(k,s_k')} - \BFe^N_{(k,s_k)}) \vert\vert_{L_2} \\
& \leq \sum_{s'} L_{D_{x(0)}\ODE} \norm{(\BFe^N_{(k,s_k)} - \BFe^N_{(k,s_k)}) }_{L_2} 
\end{align*}
Application of Cauchy-Schwarz then yields
\begin{align*}
& R_1(\ODE_{u,s}(t,.),\bX, \BFe^N_{(k,s_k')} - \BFe^N_{(k,s_k)}) \\
& \leq \sum_{k, s_k, s_k'} L_{D_{x}\ODE} \norm{(\BFe^N_{(k,s_k)} - \BFe^N_{(k,s_k)}}^2_{L_2} \abs{ \bigl( \XN_{k/N,s_k}r^{N,\text{uni}}_{k,s_k\to s_k'} + \XN_{k/N,s_k}  \sum_{s_{l}\in\calS} \int_0^1 r^{\text{pair}}_{k/N,v,s_k\to s_k', s_l}  G^N_{k/N,v} \XN_{v,s_l}dv }.
\end{align*}
By the boundedness of the rates and the finite state space, it is immediate to see that terms of the form 
\begin{align*}
\XN_{k/N,s_k}r^{\text{uni}}_{k,s_k\to s_k'} + \XN_{k/N,s_k}  \sum_{s_{l}\in\calS} \int_0^1 r^{\text{pair}}_{k/N,v,s_k\to s_k', s_l}  G^N_{k/N,v} \XN_{v,s_l}dv
\end{align*}
are bounded. Additionally,  by definition of $e^{(k,s_k)}$ which are non-zero only on a sub-interval of $(0,1]$ of size $\frac{1}{N}$, $\norm{e^{N}_{(k,s_k)} - \BFe^N_{(k,s_k)}}^2_{L_2} \leq 2/N$.
Lastly, application of the mentioned properties and bounds show that there exists a constant $C_{\tilde{R}} > 0$ such that 
$\tilde{R}_1 \bigl(\ODE_{u,s}(t,\bX)\bigr)=C_{\tilde{R}}/N=\mathcal{O}(1/N)$.
{\hfill \Halmos}
\end{proof}

\subsection{Bound on the Difference of Stochastic and Deterministic Drift (Lemma~\ref{lemma:drift_difference_bound})}

Lemma \ref{lemma:drift_difference_bound} is concerned with bounding the difference between the drift of the stochastic and deterministic system. As stated in the Lemma, the bound depends mainly on the $L_2$ difference between rates and between the graphon and graph of the stochastic system. 

\begin{lemma}\label{lemma:drift_difference_bound}
Let $D_{x} \ODE(t, \bX)(\by)$ be the directional derivative of $\bODE$ with respect to the initial condition in direction $\by$ in $L_2$. Then, \begin{align*}
\int_0^1 \sum_{s'} \left[D_{x} \ODE(t, \bX) \left( \bFG\bigl(\bODE(t, \bX) \bigr) - \bFN (\bODE(t, \bX)) \right) \right]_{v,s'} dv 
\end{align*} 
is bounded by
$6 C_{ D_{x(0)}} \left(  \frac{2 L_{\bR^{\text{pair}}} + 16C^2_{\bR^{\text{pair}}}K_{\bR^{\text{pair}}} }{N} - \frac{16C^2_{\bR^{\text{pair}}}K^2_{\bR^{\text{pair}}} }{N^2} + C_{\bR^\text{pair}}^2 \abs{\calS}^2 \opnorm{G - G^N} \right)$.
\end{lemma}

\begin{proof}{Proof.}
Implied by the finite values of $\bX \in \calX^N$
\begin{align*}
& \int_0^1 \sum_{s'} \left[D_{x} \ODE(t, \bX) \left( \bFG\bigl(\bODE(t, \bX) \bigr) - \bFN (\bODE(t, \bX)) \right) \right]_{v,s'} dv  \\
& \leq C_{N, D_{x} \ODE} \int_0^1 \sum_{s'}\left( \bFG_{v,s'}\bigl(\bODE(t, \bX) \bigr) - \bFN_{v,s'} (\bODE(t, \bX)) \right)^2dv.
\end{align*}
where $C_{N, D_{x} \ODE}$ is essentially obtained by taking the maximum over all $\bX \in \calX^N$. To conclude it remains to be shown that $\int_0^1 \left( \FG_{v,s'}\bigl(\bODE(t, \bX) \bigr) - \FN_{v,s'} (\bODE(t, \bX)) \right)^2dv$ is small. By definition we are looking at 
\begin{align}
& \int_0^1 \sum_{s'} \left( \FG_{u,s'}\bigl(\bODE(t, \bX) \bigr) - \FN_{u,s'} (\bODE(t, \bX)) \right)^2 du \notag\\
& = \int_0^1 \sum_{s'} \left( \bR^{\text{uni}}_{u,s'}\bODE_{u}(t, \bX) + \bODE_u(t, \bX)^T \int_0^1 \bR^{\text{pair}}_{u,v,s'} G_{u,v}\bODE_v(t, \bX) \ dv \right. \notag\\ 
& \quad \left. - \bR^{N, \text{uni}}_{u,s'}\bODE_u(t, \bX) - \bODE_u(t, \bX)^T \int_0^1 \bR^{N, \text{pair}}_{u,v,s'}
 G^N_{u,v}\bODE_v(t, \bX) dv \right)^2 du \notag \\
& \leq 2 \int_0^1 \sum_{s'}\left( \bigl( \bR^{\text{uni}}_{u,s'} - \bR^{N, \text{uni}}_{u,s'} \bigr) \bODE_u(t, \bX) \right)^2 du \label{eq:difference_dirft_unilateral_rates} \\
& \quad + 2 \int_0^1 \sum_{s'} \left( \bODE_u(t, \bX)^T \bigl( \int_0^1 \bR^{\text{pair}}_{u,v,s'} G_{u,v}\bODE_v(t, \bX)  dv - \int_0^1 \bR^{N, \text{pair}}_{u,v,s'} G^N_{u,v}\bODE_v(t, \bX) dv \bigr) \right)^2 \ du
\end{align}
where the inequality follows as for $f,g\in L_2(0,1]$ it holds that $\norm{f + g}^2 \leq 2( \norm{f}^2 + \norm{g}^2)$. We separately bound the two summands. The application of Lemma \ref{lemma:lipschtiz_discretization_error} directly gives a bound for the first
\begin{align*}
2\int_0^1 \sum_{s'} \left( \bigl( \bR^{\text{uni}}_{u,s'} - \bR^{N, \text{uni}}_{u,s'} \bigr) \bODE_u(t, \bX) \right)^2 du \leq 2(\frac{2 L_{\bR^{\text{uni}}} + 16C^2_{\bR^{\text{uni}}}K_{\bR^{\text{uni}}} }{N} - \frac{16C^2_{\bR^{\text{uni}}}K^2_{\bR^{\text{uni}}} }{N^2}).
\end{align*}
For the second term, first, we add and subtract the term $\bODE_u(t, \bX)^T \int_0^1 \bR^{N, \text{pair}}_{u,v,s'} G_{u,v}\bODE_v(t, \bX)  dv$ and use again the previously applied inequality to get the two terms 
\begin{align}
& 4\int_0^1 \sum_{s'}\left( \bODE_u(t, \bX)^T \bigl( \int_0^1 \bR^{\text{pair}}_{u,v,s'} -  \bR^{N, \text{pair}}_{u,v,s'}\bigr) G_{u,v}\bODE_v(t, \bX)  dv \right)^2 \text{ and } \tag{\romannumeral1} \label{eq:drift_difference_pair_rates}\\
& 4\int_0^1 \sum_{s'}\left( \bODE_u(t, \bX)^T \int_0^1 \bR^{N, \text{pair}}_{u,v,s'} \bigl(G_{u,v} - G^N_{u,v}\bigr) \bODE_v(t, \bX) \bigr) dv \right)^2. \tag{\romannumeral2} \label{eq:drift_difference_graphon_graph}
\end{align}
At its core, it remains to bound the difference of $\bR^{\text{pair}} - \bR^{N, \text{pair}}$ rate matrices and the difference $G - G^N$ between the graphon and the graph.
Similar to the bound for the unilateral rates, we use Lemma \ref{lemma:lipschtiz_discretization_error} to bound Equation \eqref{eq:drift_difference_pair_rates}, i.e., 
\begin{align*}
\eqref{eq:drift_difference_pair_rates} \leq 4 (\frac{2 L_{\bR^{\text{pair}}} + 16C^2_{\bR^{\text{pair}}}K_{\bR^{\text{pair}}} }{N} - \frac{16C^2_{\bR^{\text{pair}}}K^2_{\bR^{\text{pair}}} }{N^2}) .
\end{align*}
The second Equation \eqref{eq:drift_difference_graphon_graph} can be bounded by $4 C_{\bR^\text{pair}}^2 \abs{\calS}^2 \opnorm{G - G^N}$. To obtain this bound, we use the bounds $C_{\bR^{\text{pair}}}$ for $\bR^{N,\text{pair}}, \bR^{\text{pair}}$, the fact that $\bR^{\text{pair}}$ is a $\calS \times \calS$ matrix, and lastly that $\norm{x^G}_{L_2} = 1$. Application of the bounds yields
\begin{align*}
& \sum_{s'} \langle D_{x(0)} \ODE_{u,s}(t, \bX)_{s'}, \FG_{s'}\bigl(\bODE(t, \bX) \bigr) - \FN_{s'} (\bODE(t, \bX))\rangle_{L_2} \\
& \leq 6 C_{ D_{x(0)}} \left(  \frac{2 L_{\bR^{\text{pair}}} + 16C^2_{\bR^{\text{pair}}}K_{\bR^{\text{pair}}} }{N} - \frac{16C^2_{\bR^{\text{pair}}}K^2_{\bR^{\text{pair}}} }{N^2} + C_{\bR^\text{pair}}^2 \abs{\calS}^2 \opnorm{G - G^N} \right)
\end{align*}
which concludes the proof.
{\hfill \Halmos}
\end{proof}

\subsection{Bound on the Approximation of a Piecewise Lipschitz Function (Lemma~\ref{lemma:lipschtiz_discretization_error})}
In several part of the paper, the difference between a piecewise Lipschitz function and a discretized version needs to be quantified. The next lemma states an upper bound depending on the coarseness of the discretization as well as properties induces by the definition of piecewise Lipschitz continuity.

\begin{lemma} \label{lemma:lipschtiz_discretization_error}
Let $B$ be a bounded piecewise Lipschitz function with $K\in \N$ blocks, Lipschitz constant $L_B < \infty$ as defined in \ref{def:piecewise_lipschitz_graphon} and bound $C_B\geq 0$ such that $\abs{B(u,v)} \leq C_B$. Furthermore, let $B^N$ be its discretized version defined by $B^N(u,v) = \sum_{i,j=1}^N B(i/N,j/N)\mathds{1}_{I^N_i}(u)\mathds{1}_{I^N_j}(v)$.
Assuming that $N$ is large enough as in Definition \ref{def:large_enough_N}, it holds for $f\in L_2(0,1]$ with $\norm{f}_{L_2}\leq 1$ that
\begin{align*}
\int_0^1 \abs{\int_0^1 \bigl(B(u,v) - B^N(u,v)\bigr)f(v)dv} du \leq \frac{2 L_B + 16C_{B}^2 K}{N} - \frac{16C_{B}^2 K^2}{N^2}.
\end{align*}
Note, in the case of $B$ being Lipschitz continuous, i.e., $K=0$, the right-hand side reduces to the anticipated bound of $2 L_B/N$.
\end{lemma}

\begin{proof}{Proof.}
By Cauchy-Schwarz inequality on $L_2(0,1]$ it follows that
\begin{align*}
\int_0^1 \bigl(B(u,v) - B^N(u,v)\bigr)f(v)dv & \leq \int_0^1 \bigl(B(u,v) - B^N(u,v)\bigr)^2 dv \norm{f}^2_{L_2} \\ 
& \leq \int_0^1 \bigl(B(u,v) - B^N(u,v)\bigr)^2 dv.
\end{align*}
For the rest of the proof we refer to \cite{avella-medinaCentralityMeasuresGraphons2018} as the reasoning is identical as for their proof of Theorem 1, staring at Equation (31). The sole adaptation made is to include the case that $B$ can take values outside of $(0,1]$, leading to $2 C_B$ as the bound of the difference $\abs{B(u,v) - B^N(u,v)}$. 
\hfill \Halmos
\end{proof}

\subsection{Taylor Expansion with Remainder for Banach Spaces (Lemma~\ref{thrm:taylor_with_remainder_for_banach_spaces})} \label{sec:taylor_expansion_banach}

For completeness, we reformulate the result from \cite{blanchardMathematicalMethodsPhysics2015} pp. 524-525.

\begin{lemma}\label{thrm:taylor_with_remainder_for_banach_spaces}
Suppose $E, F$ are real Banach spaces, $U \subset E$ an open and nonempty subset and $f\in \mathcal{C}^n(U,F)$ ($n$-times continuously differentiable). Given $x_0 \in U$ choose $r>0$ such that $x_0 + B_r \subset U$, where $B_r$ is the open ball in $E$ with center $0$ and radius $r$. Then for all $\ell\in B_r$ we have, using the abbreviation $(\ell)^k=(\ell,...,\ell)$, 
\begin{align}
f(x_0 + \ell) = \sum_{k=0}^n \frac{1}{k!}f^{(k)}(x_0)(\ell)^k + R_n(f,x_0,\ell), \label{eq:taylor_expansion_definition}
\end{align}
where the remainder $R_n$ has the form 
\begin{align}
R_n\left(f, x_0 , \ell\right)=\frac{1}{(n-1) !} \int_0^1(1-\theta)^{n-1}\left[f^{(n)}\left(x_0+\theta \ell\right)-f^{(n)}\left(x_0\right)\right](\ell)^n d\theta. \label{eq:tayloer_remainder_definition}
\end{align}
\end{lemma}

\end{appendices}

\end{document}